\documentclass[12pt]{article}
\usepackage[dvips]{epsfig}
\usepackage{amsthm,amsmath,amssymb,mathtools}
\usepackage{ascmac}
\usepackage{latexsym}
\usepackage{graphics}
\usepackage{url}
\usepackage{placeins}
\usepackage{hyperref}
\usepackage{color}
\usepackage{bm}
\usepackage{comment}
\usepackage{statex}
\usepackage{diagbox}
\usepackage{makecell}
\usepackage{multirow}

\definecolor{blue}{rgb}{0,0,0.9}
\definecolor{red}{rgb}{0.9,0,0}
\definecolor{green}{rgb}{0,0.9,0}
\definecolor{violet}{rgb}{0.5804,0.0000,0.8275}

\def\@themcountersep{}

\newtheorem{assum}{Assumption}[section]

\newtheorem{theorem}{Theorem}[section]

\newtheorem{THEO}{Theorem}[section]
\newtheorem{ALGo}[THEO]{Algorithm}
\newtheorem{CONJ}[THEO]{Conjecture}
\newtheorem{COND}[THEO]{Condition}
\newtheorem{ASSUMP}[THEO]{Assumption}
\newtheorem{CORO}[THEO]{Corollary}
\newtheorem{DEFI}[THEO]{Definition}
\newtheorem{EXAMP}[THEO]{Example}
\newtheorem{FACT}[THEO]{Fact}
\newtheorem{HYPO}[THEO]{Hypothesis}
\newtheorem{LEMM}[THEO]{Lemma}
\newtheorem{PROB}[THEO]{Problem}
\newtheorem{PROP}[THEO]{Proposition}
\newtheorem{REMA}[THEO]{Remark}
\newcommand{\theo}{\begin{THEO}}
\newcommand{\algo}{\begin{ALGo} \rm}
\newcommand{\cond}{\begin{COND} \rm}
\newcommand{\assump}{\begin{ASSUMP} \rm}
\newcommand{\conj}{\begin{CONJ}}
\newcommand{\coro}{\begin{CORO}}
\newcommand{\defi}{\begin{DEFI} \rm}
\newcommand{\examp}{\begin{EXAMP} \rm}
\newcommand{\fact}{\begin{FACT}}
\newcommand{\hypo}{\begin{HYPO} \rm}
\newcommand{\lemm}{\begin{LEMM}}
\newcommand{\prob}{\begin{PROB} \rm}
\newcommand{\prop}{\begin{PROP}}
\newcommand{\rema}{\begin{REMA} \rm}
\newcommand{\etheo}{\end{THEO}}
\newcommand{\ealgo}{\end{ALGo}}
\newcommand{\econd}{\end{COND}}
\newcommand{\eassump}{\end{ASSUMP}}
\newcommand{\econj}{\end{CONJ}}
\newcommand{\ecoro}{\end{CORO}}
\newcommand{\edefi}{\end{DEFI}}
\newcommand{\eexamp}{\end{EXAMP}}
\newcommand{\efact}{\end{FACT}}
\newcommand{\ehypo}{\end{HYPO}}
\newcommand{\elemm}{\end{LEMM}}
\newcommand{\eprob}{\end{PROB}}
\newcommand{\eprop}{\end{PROP}}
\newcommand{\erema}{\end{REMA}}

\def\0{\mbox{\bf 0}}
\def\1{\mbox{\bf 1}}
\def\2{\mbox{\bf 2}}
\def\3{\mbox{\bf 3}}
\def\4{\mbox{\bf 4}}
\def\5{\mbox{\bf 5}}
\def\6{\mbox{\bf 6}}
\def\7{\mbox{\bf 7}}
\def\8{\mbox{\bf 8}}
\def\9{\mbox{\bf 9}}
\def\a{\mbox{\boldmath $a$}}
\def\b{\mbox{\boldmath $b$}}

\def\p{\mbox{\boldmath $p$}}
\def\q{\mbox{\boldmath $q$}}

\def\s{\mbox{\boldmath $s$}}

\def\u{\mbox{\boldmath $u$}}
\def\v{\mbox{\boldmath $v$}}
\def\w{\mbox{\boldmath $w$}}
\def\x{\mbox{\boldmath $x$}}

\def\B{\mbox{\boldmath $B$}}
\def\C{\mbox{\boldmath $C$}}

\def\H{\mbox{\boldmath $H$}}
\def\I{\mbox{\boldmath $I$}}

\def\L{\mbox{\boldmath $L$}}

\def\O{\mbox{\boldmath $O$}}

\def\Q{\mbox{\boldmath $Q$}}

\def\S{\mbox{\boldmath $S$}}

\def\W{\mbox{\boldmath $W$}}

\def\DC{\mbox{$\cal D$}}

\def\SC{\mbox{$\cal S$}}

\def\Real{\mbox{$\mathbb{R}$}}

\def\SymMat{\mbox{$\mathbb{S}$}}

\def\bbeta{\mbox{\boldmath $\beta$}}

\def\blam{\mbox{\boldmath $\lambda$}}
\def\bell{\mbox{\boldmath $\ell$}}

\begin{document}

\title{ \Large Solving  Pooling Problems by LP and SOCP Relaxations and Rescheduling Methods
} 

\author{
\normalsize 
Masaki Kimizuka\thanks{Department of  Mathematical and Computing Science,
            Tokyo Institute of Technology, 2-12-1 Oh-Okayama, Meguro-ku, Tokyo 152-8552, Japan        
             ({\tt kimi3masa0@gmail.com}). },   \and \normalsize
Sunyoung Kim\thanks{Department of Mathematics, Ewha W. University, 52 Ewhayeodae-gil, Sudaemoon-gu, Seoul 03760, Korea 
			({\tt skim@ewha.ac.kr}). The research was supported
               by  NRF 2017-R1A2B2005119.}, \and \normalsize
Makoto Yamashita\thanks{Department of  Mathematical and Computing Science,
            Tokyo Institute of Technology, 2-12-1 Oh-Okayama, Meguro-ku, Tokyo 152-8552, Japan        
             ({\tt makoto.yamashita@is.titech.ac.jp}).
	This research was partially supported by JSPS KAKENHI (Grant number: 15k00032).
 	}
}

\date{\normalsize February, 2018}

\maketitle 

\begin{abstract}
\noindent
The pooling problem is an important industrial problem in the class of
network flow problems for allocating gas flow in pipeline transportation networks. 
For P-formulation of the pooling problem with time discretization,
we propose second order cone programming (SOCP)  and linear programming (LP) relaxations
 and prove that they obtain the same
optimal value as the semidefinite programming relaxation. 
The equivalence  among the optimal values of the three relaxations is also computationally shown.
Moreover, a rescheduling method is proposed
to efficiently refine the solution obtained by the SOCP or  LP relaxation.
 The efficiency of the SOCP and the LP relaxation and
the proposed rescheduling method is illustrated with numerical results on the test instances from the work of
Nishi in 2010, some large instances,  and Foulds 3, 4, 5 test problems.

\end{abstract}

 



\section{Introduction}

The pooling problem is a network flow problem 
for allocating gas flow in pipeline transportation networks with minimum cost. It arises from applications in the petroleum industry.
Networks of the pooling problem have three types of nodes: sources, blending tanks called pools and  plants.
Gas flows from the sources are blended in the blending tanks and plants  to
produce the desired final products. 
To model such blending,
the pooling problem is formulated as  bilinear nonconvex optimization problems, thus nonconvex quadratically
constrained quadratic problems  (QCQPs)  \cite{HAVERLY78}.

The pooling problem has been studied in two main formulations, the P-formulation \cite{HAVERLY78} and 
Q-formulation \cite{BENTAL94}. 
Difficulties of solving the pooling problem
arise from the existence of pipeline constraints  formulated with binary variables and the blending process
represented as nonlinear constraints. 
The resulting problem is a mixed-integer nonlinear program
known as NP-hard  \cite{ALFAKI13}. 
Various solution methods including   
 approximating heuristics, linear programming relaxations,
decomposition techniques  have been proposed for these formulations.  In particular, 
successive linear relaxations \cite{FLOUDAS09,LIBERTI06} 
 approximate the pooling problem based on
a first-order Taylor expansion.
Another  popular approach is branch-and-bound algorithms which have been implemented to solve
 large-scale problems  \cite{ALFAKI13}.

While
conic relaxations methods including
semidefinite programming (SDP), second order cone (SOCP)  and linear programming (LP) relaxations of general nonconvex QCQPs
have  been widely used to approximate the optimal values of the problems, they have not studied extensively for the pooling problem.
In fact,  no literature on  SOCP and LP relaxations
of the pooling problem could be found to the authors'  best knowledge.  SDP relaxations of the pooling problem were studied in \cite{NISHI10,ZANNI13}.
SDP relaxations of nonconvex QCQPs are known to
 provide tighter bounds for the optimal value than SOCP and LP relaxations, however, solving SDP relaxations  by the primal-dual interior-point methods  \cite{STURM99,TOH98,yamashita2012latest,yamashita2003implementation} is computationally expensive. Thus,
 the size of the problems that can be solved by SDP relaxations remains very limited. 
 From a computational perspective,
SOCP and LP relaxations are more efficient than SDP relaxations, as a result, large-sized problems can be solved by SOCP and LP
relaxations \cite{KIM2001}.

The main purpose of this paper is to propose
an efficient computational method that employs LP   and SOCP relaxations and a rescheduling method for the 
P-formulation of
the pooling problem with time discretization. The  LP relaxation of nonconvex QCQPs in this
paper is different from the linear programming based on a first-order Taylor expansion \cite{FLOUDAS09,LIBERTI06}.  
Let  $\zeta^*$ be the optimal value of a general  nonconvex QCQP that minimizes the objective function.
Among $\zeta^*$ and the optimal values of SDP, SOCP and LP relaxations of the QCQP,  we have the following relationship:
\[ \zeta_{LP}^* \leq \zeta_{SOCP}^* \leq \zeta_{SDP}^* \leq \zeta^*. \]
LP relaxations are known to be most  efficient  and SDP relaxations 
 most time-consuming among the three relaxations  for solving general QCQPs.
For our formulation of the pooling problem, we prove that the SDP, SOCP and LP relaxations   provide the equivalent
optimal value:
\begin{equation}
\zeta_{LP}^* = \zeta_{SOCP}^* = \zeta_{SDP}^* \leq \zeta^*.  \label{EQUI}
\end{equation} 
More precisely, the LP relaxation can be used to obtain the same quality of the optimal value as that of the SDP relaxation 
with much less computational efforts. 
Thus, larger pooling problems can be handled with the LP relaxation.
We theoretically prove the equivalence \eqref{EQUI} and present the computational results that support \eqref{EQUI}.
Moreover, we demonstrate that \eqref{EQUI} holds for other formulations of the pooling problem where
bilinear terms appear with no squared terms of the variables.  The LP relaxation presented in this paper can be
used to efficiently solve  different formulations of the pooling problem.

The solution obtained by the LP, SOCP and SDP relaxations of the pooling problem should be refined 
to satisfy all the constraints of the original problem. For this issue, Nishi \cite{NISHI10} proposed a method  that 
solves  a mixed-integer linear program
using the  solution obtained by the SDP relaxation, then applies an iterative procedure for the nonlinear terms, and finally
uses a nonlinear program solver to attain a local optimal solution.  In the three steps of his method,  
the nonlinear program 
solver particularly takes long computational time,  making the entire  method very time-consuming.
To reduce the computational burden caused by applying a nonlinear program solver, we propose
a rescheduling method which successively 
updates a local optimal solution by applying the SOCP or LP relaxation 
to partial  time steps of the entire time discretization.
The proposed technique significantly increases the computational efficiency of the entire method, which enables us
to solve large pooling problems. For one test instance with 1228 variables,
the rescheduling method reduced the computational time for a general nonlinear solver by  1/60.  

This paper is organized as follows: Section 2  briefly describes the pooling problem. In Section 3, we illustrate SDP, SOCP,  
and LP relaxations of the  general nonconvex QCQPs.
Section 4 includes the proof of the optimal values on the three relaxations of our formulation of the pooling problem
and discusses how the result can be applied to other formulations of the pooling problem.
In Section 5, we  
describe the proposed rescheduling methods in detail.
Section 6 presents numerical results on the test problems in \cite{NISHI10}, some large instances
and Foulds 3, 4, 5. We conclude in Section 7.


\section{The pooling problem}
We first describe the formulation of the pooling problem  with time discretization in \cite{NISHI10}, then
 present our formulation.

\begin{figure}[hbt]
\begin{center}
\vspace*{-3cm}
\includegraphics[width=13cm, height=13cm]{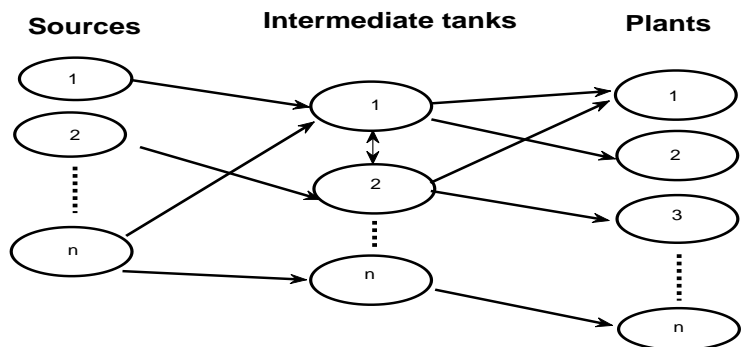} \label{fig2.1}
\end{center} \vspace*{-5.7cm}
\caption{Overview}
\end{figure}

\subsection{Notation} 
Let $M_S$, $M_I$ and $M_P$  denote the numbers of sources, intermediate tanks, and plants, respectively.
We also let $V$ be the set of all nodes.
The sets of sources, intermediate tanks and plants are denoted by $V_S$, $V_I$ and $V_P$,
respectively, as follows:
\begin{eqnarray*}
&V_S& = \{ 1, 2, \dots , M_S\}, \ \ 
V_I = \{ M_S+1, \dots , M_S+M_I\}, \  \\
&V_P& = \{ M_S+M_I+1, \dots , M_S+M_I+M_P\}, \ 
V =  V_S \cup V_I \cup V_P.
\end{eqnarray*}
The arrows in Figure \ref{fig2.1} mean pipelines.
The pipeline between $i$ and $j \in V$ is denoted as $(i, j)$, and 
the set of pipelines is denoted as $A$.
Furthermore, for the $i$th node, the set of entering nodes and that of leaving nodes are denoted, respectively, as
\begin{eqnarray*}
I(i) = \{ j \in V | (j,i) \in A\} , \ 
E(i) = \{ k \in V | (i, k) \in A\}.
\end{eqnarray*}
We use $M_T$ to mean the number of time discretization 
and each time slot can be identified by 
$t \in T = \{ 1, \dots, M_T\}$.

As our formulation is based on time discretization,  $p^{t}_{i}$  denotes the quantity stored in $i \in V$, 
and $q^{t}_{i}$  the quality of the $i$th node  at time $t$.
The flow in the pipeline $(i, j)$ is denoted as $a^{t}_{ij}$, 
and  binary variable $u^{t}_{ij}$ means whether the pipeline $(i, j)$ is used at time $t$ or not.
We also introduce $v^{t}_{i} \ (i \in V_P)$ to evaluate the quality shortage for 
the requirement at the plant $i \in V_P$, at time $t$.
The variables, constants and sets are summarized in Table~\ref{T-1}.

\subsection{Problem formulation with time discretization}


\begin{figure}[hbt]
\begin{center}
\vspace*{-3.3cm}
\includegraphics[width=13cm, height=13cm]{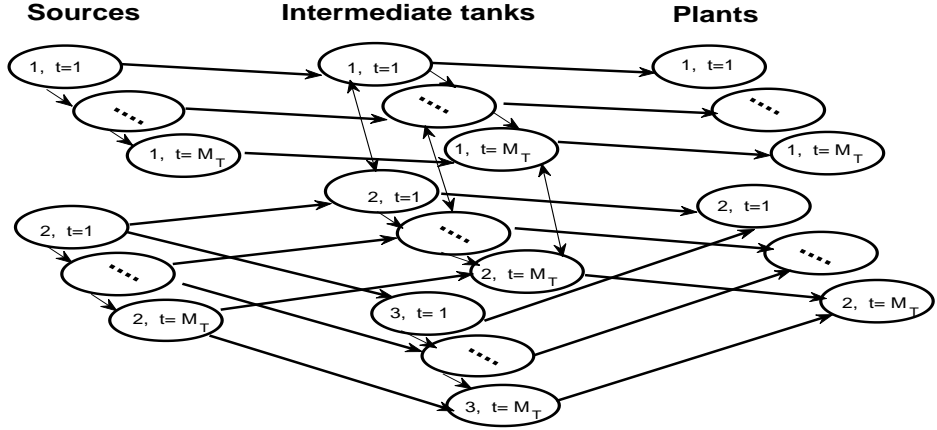}
\end{center} \vspace*{-4.5cm}
\caption{The pooling problem with time discretization}
 \label{fig2.22}
\end{figure}

The pooling problem has been represented with various  formulations.
We formulate P-formulation
with time discretization as a nonconvex mixed-integer QCQP.
As shown in Figure \ref{fig2.22}, the pooling problem with time discretization can be viewed as
a network with the arcs connecting sources, intermediate tanks, and plants at the same time step.

The objective function of the pooling problem can be modeled as follows:
\begin{eqnarray*}
\displaystyle \min_{a,p,q,v} \displaystyle \sum_{t \in T} \sum_{(i,j) \in A} CA_{ij} a_{ij}^{t} + 
 \sum_{t \in T} \sum_{i \in V_{P}} CQ_{i}RC_{i}^{t} v_{i}^{t}.
\end{eqnarray*}
Here $CA_{ij}$ is the  transportation cost for the pipeline $(i, j)$, $CQ_{i}$ the penalty cost for the shortage
 at the $i$th node,
and $RC_{i}^{t}$  the required quantity at the $i$th node.
The first  and second terms of the objective function represent the transportation cost
and the penalty cost. 

\vspace{-0.5cm}
\begin{table}[h!t]
\begin{center}
\caption{The sets, constants and variables of P formulation} \label{T-1}
\begin{tabular}{|c|c||c|c|}
\hline
\multicolumn{4}{|c|}{Sets}\\
\hline
$V_S$&the set of sources &$V_I$&the set of intermediate tanks \\
\hline
$V_P$&the set of plants&$(i,j)$&the  pipeline between $i$ and $j$\\
\hline
$I(i)$&\multicolumn{3}{c|}{the set of entering nodes to the $i$th nodes}\\
\hline
$E(i)$&\multicolumn{3}{c|}{the set of leaving nodes from the $i$th nodes}\\
\hline
\multicolumn{4}{|c|}{Constants}\\
\hline
$M_S$&the number of sources &$M_I$&the number of intermediate tanks \\
\hline
$M_P$&the number of plants&$M_{T}$&the number of time discretization \\
\hline
$p^{\min}_{i}$& the minimum quantity&$p^{\max}_{i}$&the maximum quantity\\
\hline
$SA^{t}_{i}$& the supply quantity&$SQ^{t}_{i}$&the supply quality\\
\hline
$U_{ij}$&the maximum flow&$L_{ij}$&the minimun flow \\
\hline
$RC^{t}_{i}$&the required quantity&$RQ^{t}_{i}$&the required quality\\
\hline
$CA_{ij}$&the transportation cost for $(i, j)$&$CQ_{i}$&the penalty cost \\
\hline
\multicolumn{4}{|c|}{Variables for the $i$th node at time $t$}\\
\hline
$a_{ij}^{t}$&flow in the pipeline $(i, j)$&$p_{i}^{t}$& the quantity \\
\hline
$q_{i}^{t}$&the quality &$u_{ij}^{t}$& binary variables \\
\hline
$v_{i}^{t}$& the quality shortage&\multicolumn{2}{c|}{{\text{ }}} \\
\hline
\end{tabular}
\end{center}
\end{table}

For constraints,
 $v_{i}^{t}$ is introduced  to denote the shortage in quality at the $i$th node.
If  $RQ_{i}^{t}$ is used to denote the required quality at the $i$th node, then
\begin{eqnarray*}
v_{i}^{t} = \max \{ 0, RQ_{i}^{t} - q_{i}^{t} \} \ \ \ (i \in V_{P} \ , \ t \in T).
\end{eqnarray*}
At each time $t$, each node can be connected to at most one pipeline, therefore
we must have 
\begin{eqnarray*}
u_{ij}^{t} \in \{ 0, 1\}, \  \  
\sum_{j \in I(i)}  u_{ji}^{t} +\sum_{k \in E(i)} u_{ik}^{t} \leq 1
 \ \ (i \in V,  \ (i,j) \in A, \ t \in T = \{ 1, \dots , M_T\}).
\end{eqnarray*}
The flow of each pipeline has  a lower and upper bound,
\begin{eqnarray*}
u_{ij}^{t}L_{ij} \leq a_{ij}^{t} \leq u_{ij}^{t}U_{ij} \ ((i,j) \in A , t \in T),
\end{eqnarray*}
where $L_{ij}$ and $U_{ij}$ are the lower bound and upper bound for the flow in the pipeline $(i, j)$.

Two constraints can  be derived from mixing two kinds of oil with different quantity $p_{i}^{t}$ and  quality $q_{i}^{t}$. %
For instance,
we consider mixing oil $1$ and $2$ to produce new oil $3$ in Figure \ref{fig2.2}. 
Assume that each node has the quantity, $p_{1}, p_{2}$ and $p_{3}$ and the quality, $q_{1}, q_{2}$ and $q_{3}$.
\begin{figure}[h!bt]
\begin{center}
\includegraphics[width=7cm, height=2.cm]{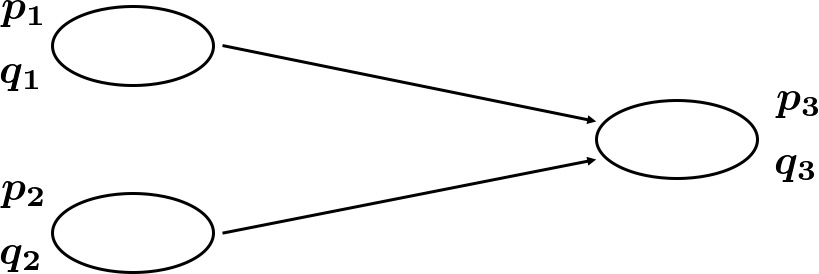} 
\end{center} \vspace*{-0.8cm}
\caption{Flow of mixing oil} \label{fig2.2}
\end{figure}
The constraint for the amount of new oil $p_{3}$ is that it
should be equivalent to the sum of  oil 1 and 2, 
{\it i.e.,} $p_{3} = p_{1} + p_{2}$. 
Second, the quality of  new oil $q_{3}$ should be computed by the weighted average of oil 1 and 2,
{\it i.e.,} $p_{3}q_{3} = p_{1}q_{1} + p_{2}q_{2}$ holds. 
Thus, necessary constraints for the pooling problem are
\begin{eqnarray*}
  p^{t+1}_{i} = p_{i}^{t} + SA_{i}^{t} - \sum_{k \in E(i)} a_{ik}^{t}, \  \ p_{i}^{t} \geq 0, \  \ p_i^{M_{T}+1} = 0 \  \
\ \ (i \in V_S, t \in T ),  \\
  p^{t+1}_{i} q_{i}^{t+1} = p_{i}^{t} q_{i}^{t}+ SA_{i}^{t}SQ_{i}^{t} - \sum_{k \in E(i)} a_{ik}^{t} q_{i}^{t} \ 
 \ \ (i \in V_S, t \in T), \\
  p^{t+1}_{i} = p_{i}^{t} + \sum_{k \in I(i)} a_{ki}^{t} - \sum_{k \in E(i)} a_{ik}^{t},  
 \  \  p_{i}^{\min} \leq p_{i}^{t} \leq p_{i}^{\max} \ 
 \ \ (i \in V_I, t \in T),  \\
  p^{t+1}_{i} q_{i}^{t+1} = p_{i}^{t} q_{i}^{t} + \sum_{k \in I(i)}a_{ki}^{t} q_{k}^{t} 
 -  \sum_{k \in E(i)} a_{ik}^{t} q_{i}^{t} \  \ \ (i \in V_I , t \in T), \\
 q_{i}^{t}= \frac{1}{RC_{i}^{t}} \sum_{j \in I(i)} a_{ji}^{t} q_{j}^{t} \ 
  \ \ (i \in V_{P}, t \in T).
\end{eqnarray*}
Here, $SA_i^t$ and $SQ_i^t$ are the supplied quantity and quality at the source  $i \in V_S$.
The constraint $p_i^{M_{T}+1} = 0$ for $ i \in V_S$ requires the quantity at the sources
should be empty at time $M_{T+1}$.

We describe the formulation of the pooling problem in \cite{NISHI10}   as follows:
\begin{eqnarray*}
(PP) &\displaystyle \min_{a,p,q,u,v}&\displaystyle \sum_{t \in T} \sum_{(i,j) \in A} CA_{ij}a_{ij}^{t} + 
 \sum_{t \in T} \sum_{i \in V_{P}} CQ_{i}RC_{i}^{t}v_{i}^{t} \\
 &{\text{subject to}}&  u_{ij}^{t}L_{ij} \leq a_{ij}^{t} \leq u_{ij}^{t}U_{ij} \ \  ((i,j) \in A, t \in T),
 \\
&&u_{ij}^{t} \in \{ 0, 1 \}, \  \  \sum_{j \in I(i)}  u_{ji}^{t} +\sum_{k \in E(i)}  u_{ik}^{t} \leq 1
 \  \ (i \in V, \ (i,j) \in A, \ t \in T), \\
 && p^{t+1}_{i} = p_{i}^{t} + SA_{i}^{t} - \sum_{k \in E(i)}a_{ik}^{t}, \  \ p_{i}^{t} \geq 0 \ , \ p^{M_{T}+1} = 0 \  \
 \ (i \in V_S, t \in T),  \\
 && p^{t+1}_{i}q_{i}^{t+1} = p_{i}^{t} q_{i}^{t}+ SA_{i}^{t}SQ_{i}^{t} - \sum_{k \in E(i)}a_{ik}^{t}q_{i}^{t} \  
  \ (i \in V_S, t \in T), \\
  && p^{t+1}_{i} = p_{i}^{t} + \sum_{k \in I(i)}a_{ki}^{t} - \sum_{k \in E(i)}a_{ik}^{t},  
 \  \ p_{i}^{\min} \leq p_{i}^{t} \leq p_{i}^{\max} \ 
  \ (i \in V_I, t \in T),  \\
 && p^{t+1}_{i}q_{i}^{t+1} = p_{i}^{t} q_{i}^{t} + \sum_{k \in I(i)}a_{ki}^{t}q_{k}^{t} 
 -  \sum_{k \in E(i)}a_{ik}^{t}q_{i}^{t} \  \ \ (i \in V_I, t \in T), \\
 &&q_{i}^{t} = \frac{1}{RC_{i}^{t}} \sum_{j \in I(i)} a_{ji}^{t}q_{j}^{t} \ 
  \ (i \in V_{P}, t \in T),\\
 && v_{i}^{t} \geq  \max \{0, RQ_{i}^{t}-q_{i}^{t} \} \   \ (i \in V_{P}, t \in T).
\end{eqnarray*}
Note that some of the above constraints are quadratic and nonconvex.
With nonconvex constraints and the binary variables
$u_{ij}^{t} \in \{0,1\}$, 
the  formulation (PP) of the pooling problem is 
a nonconvex mixed-integer nonlinear programming problem. 
Each quadratic term of the formulation of the pooling problem is always bilinear, and 
no squared terms of variables appear in the constraints.
Even if $u$ can be removed, the problem is nonconvex,
 as a result, it is difficult to apply an existing mixed-integer nonlinear programming  method.
It is known that global optimum solutions cannot be obtained within reasonable time,
since the pooling problem has been shown to be NP-hard \cite{ALFAKI12}.

\subsection*{Eliminating binary variables}
In \cite{NISHI10},  the pipeline constraints were modified 
to remove the binary variables  $u$  before applying the SDP relaxation problem.
We  briefly describe the elimination of the binary variables. 
The constraints involving the binary variables  were rewritten with $a^t_{ij}$  using the relation between
 $u^t_{ij}$ and $a^t_{ij}$. 
More precisely,  the constraints given by
\begin{eqnarray*}
&&\sum_{j \in I(i)}  u_{ji}^{t} +\sum_{k \in E(i)}  u_{ik}^{t} \leq 1
 \  \ (i \in V, \  \ t \in T),  \ Ê\  u_{ij}^{t} \in \{ 0, 1 \}  \  
  \ ((i,j) \in A, \  \ t \in T).
\end{eqnarray*}
require  that 
 at most one pipeline for all $i \in V$ and $t \in T$ should be used.
Thus, an equivalent constraint can be described in terms of $\a$ as follows:
\begin{eqnarray*} 
\sum_{j,k \in I(i), j \not= k} a_{ji}^{t}a_{ki}^{t} +
 \sum_{j,k \in E(i), j \not= k} a_{ij}^{t}a_{ik}^{t} +\sum_{j \in I(i), k \in E(i)} a_{ji}^{t}a_{ik}^{t} = 0
 \  \ (i \in V, \  \ t \in T).
\end{eqnarray*}
To remove the binary variables from 
$u_{ij}^{t} L_{ij} \leq a_{ij}^{t} \leq u_{ij}^{t}U_{ij}  \ ((i,j) \in A, \  \ t \in T)$,
the lower bound on $a^{t}_{ij}$ was modified to the following nonnegativity,
\begin{eqnarray}
0 \leq a^{t}_{ij} \leq U_{ij}  \  \ ((i,j) \in A, \  \ t \in T).  \label{RELAXBOUND}
\end{eqnarray}
As a result, the following problem is derived:
\begin{eqnarray}
&\displaystyle \min_{a,p,q,v}&\displaystyle \sum_{t \in T} \sum_{(i,j) \in A} CA_{ij}a_{ij}^{t} + 
 \sum_{t \in T} \sum_{i \in V_{P}} CQ_{i}RC_{i}^{t}v_{i}^{t} \nonumber \\
 &{\text{subject to}}&  
 \sum_{j,k \in I(i), j \not= k} a_{ji}^{t}a_{ki}^{t} +
 \sum_{j,k \in E(i), j \not= k} a_{ij}^{t}a_{ik}^{t} +\sum_{j,k \in I(i), j \not= k} a_{ji}^{t}a_{ki}^{t} = 0
 \  \ (i \in V, \  \ t \in T), \nonumber \\
&&  0   \leq a_{ij}^{t} \leq U_{ij}, \ \  \ ((i,j) \in A, \  \ t \in T)  \nonumber \\
 && p^{t+1}_{i} = p_{i}^{t} + SA_{i}^{t} - \sum_{k \in E(i)}a_{ik}^{t} \ , \ p_{i}^{t} \geq 0 \ , \ p^{M_{T}+1} = 0 \  \
\ \ (i \in V_S, t \in T),  \nonumber \\
 && p^{t+1}_{i}q_{i}^{t+1} = p_{i}^{t} q_{i}^{t}+ SA_{i}^{t}SQ_{i}^{t} - \sum_{k \in E(i)}a_{ik}^{t}q_{i}^{t} \ 
 \ \ (i \in V_S, t \in T), \nonumber \\
 && p^{t+1}_{i} = p_{i}^{t} + \sum_{k \in I(i)}a_{ki}^{t} - \sum_{k \in E(i)}a_{ik}^{t} \ 
 \ \  p_{i}^{\min} \leq p_{i}^{t} \leq p_{i}^{\max} \ 
 \ \ (i \in V_I, t \in T),  \nonumber \\
 && p^{t+1}_{i}q_{i}^{t+1} = p_{i}^{t} q_{i}^{t} + \sum_{k \in I(i)}a_{ki}^{t}q_{k}^{t}, 
 -  \sum_{k \in E(i)}a_{ik}^{t}q_{i}^{t} \  \ \ (i \in V_I, t \in T), \nonumber \\
 &&q_{i}^{t} = \frac{1}{RC_{i}^{t}} \sum_{j \in I(i)} a_{ji}^{t}q_{j}^{t} \ 
  \ \ (i \in V_{P}, t \in T),\nonumber \\
 && v_{i}^{t} \geq \max \{0, RQ_{i}^{t}-q_{i}^{t} \} \
 \ \ (i \in V_{P}, t \in T). \label{PPM}
\end{eqnarray}
While the number of constraints in  \eqref{PPM} is the same as the number
of constraints of the original pooling problem, the number of variables in \eqref{PPM}  is small   compared to
 the original pooling problem.
Thus, \eqref{PPM} can be solved more efficiently by conic relaxation methods than the original pooling problem.

\subsection{The proposed formulation}

Although the modified problem \eqref{PPM} in \cite{NISHI10} has  reduced the
 number of variables and no binary variables,
 \eqref{PPM} may not have an  interior point. 
 If SDP relaxations are used to solve problems with no interior point, as in \cite{NISHI10}, SDP solvers
 based on primal-dual interior-pont methods \cite{STURM99,TOH98,yamashita2012latest,yamashita2003implementation} frequently fail due to 
 numerical instability. To avoid such numerical difficulty, 
 we
relax the equality to inequalities. 
For instance,
we first transform  equality constraints of the form
$a^T x = b$ into $-\lambda \leq a^T x - b \leq \lambda$ introducing by
a new variable $\lambda$. Then,
we add $\lambda$ to the objective function as a penalty function.

Our formulation of the pooling problem is:
\begin{eqnarray}
&\mbox{ }& \nonumber \\
&\displaystyle \min_{a,p,q,v}&\displaystyle \sum_{t \in T} \sum_{(i,j) \in A} CA_{ij}a_{ij}^{t} + 
 \sum_{t \in T} \sum_{i \in V_{P}} CQ_{i}RC_{i}^{t}v_{i}^{t} + \delta \sum_{t \in T} \sum_{(i,j) \in A} \lambda_i^t  \nonumber \\
 &{\text{subject  to}}&  -\lambda_i^t \leq 
 \sum_{j,k \in I(i), j \not= k} a_{ji}^{t}a_{ki}^{t} +
 \sum_{j,k \in E(i), j \not= k} a_{ij}^{t}a_{ik}^{t} +\sum_{j,k \in I(i), j \not= k} a_{ji}^{t}a_{ki}^{t} \leq \lambda_i^t
 \  \ (i \in V, \  \ t \in T), \nonumber \\
&&  0   \leq a_{ij}^{t} \leq U_{ij} \ \  \ ((i,j) \in A, \  \ t \in T),  \nonumber \\
 &&-\lambda_i^t  \leq - p^{t+1}_{i} +  p_{i}^{t} + SA_{i}^{t} - \sum_{k \in E(i)}a_{ik}^{t}  \leq \lambda_i^t, \  \ p_{i}^{t} \geq 0, \  \ p^{M_{T}+1} =0 \  \
\ \ (i \in V_S, t \in T),  \nonumber \\
 && -\lambda_i^t  \leq - p^{t+1}_{i}q_{i}^{t+1} + p_{i}^{t} q_{i}^{t}+ SA_{i}^{t}SQ_{i}^{t} - \sum_{k \in E(i)}a_{ik}^{t}q_{i}^{t}  \leq \lambda_i^t \  
 \ \ (i \in V_S,\ t \in T), \nonumber \\
 && -\lambda_i^t  \leq -p^{t+1}_{i} + p_{i}^{t} + \sum_{k \in I(i)}a_{ki}^{t} - \sum_{k \in E(i)}a_{ik}^{t}  \leq \lambda_i^t,  \ 
 \  \ p_{i}^{\min} \leq p_{i}^{t} \leq p_{i}^{\max} \ 
 \ \ (i \in V_I, \  t \in T),  \nonumber \\
 && -\lambda_i^t \leq -p^{t+1}_{i}q_{i}^{t+1} + p_{i}^{t} q_{i}^{t} + \sum_{k \in I(i)}a_{ki}^{t}q_{k}^{t} 
 -  \sum_{k \in E(i)}a_{ik}^{t}q_{i}^{t}  \leq \lambda_i^t  \  \ \ (i \in V_I, \ t \in T), \nonumber \\
 &&-\lambda_i^t \leq -q_{i}^{t} + \frac{1}{RC_{i}^{t}} \sum_{j \in I(i)} a_{ji}^{t}q_{j}^{t}  \leq \lambda_i^t  \ 
  \ \ (i \in V_{P}, \ t \in T), \nonumber \\
 && v_{i}^{t} \geq \max \{0, RQ_{i}^{t}-q_{i}^{t} \} \
  \ \ (i \in V_{P}, \ t \in T), \  \ \lambda_i^t \geq 0 \ (i \in V, \ t \in T),   \label{PPO}
\end{eqnarray}
where $\delta$ is a penalty parameter. 

For the subsequent discussion,
we express \eqref{PPO} using variable $\x$ defined as
$ \x = \{ \a, \p, \q, \v \}$.
More precisely, each set of variables are ordered in the following order.
%
\begin{eqnarray*}
&\a& = \{ a^{1}, \dots , a^{M_T}\}, \  \ \\
&\a^{t}& =\{ a^{t}_{ij} \| (i,j) \in A \} \  \ (t \in T), \\
%
&\p& = \{ p^{2}, \dots, p^{M_T}, p^{M_T+1} \}, \  \ \\
&\p^{t}& =  \{ p^{t}_{i}, \ p^{t}_{j} \|  i \in  V_{S} ,\ j \in  V_{I}  \}  \  \ (t \in T \ \backslash \{ 1 \}), \  \ \\
&\p^{M_T+1}& = \{ p^{M_T+1}_{i} \| i \in V_{I} \}, \  \   \\
%
&\q& = \{ q^{1}, \dots, q^{M_T+1} \}, \  \ \\
&\q^{1}& = \{ q^{1}_{i} \| i \in V_{P} \} ,\  \  \  \ \\
&\q^{t}& =  \{ q^{t}_{i},\ q^{t}_{j} ,\ q^{t}_{k} \|  i \in V_{S} ,\ j \in V_{I} ,\ k \in V_{P} \}
 \  \ (t \in T \ \backslash \{ 1\}), \  \ \\
&\q^{M_T+1}& = \{ q^{M_T+1}_{i} \| i \in V_{I}\}, \  \   \\
%
&\v& = \{ v^{1}, \dots, v^{M_T} \}, \  \ \\
&\v^{t}& = \{ v^{t}_{i} \| i \in V_{P} \} \  \ (t \in T).
\end{eqnarray*}

Let $n$ be the length of $\x$, that is, $\x \in \Real^n_+$ where $\Real^n_+$ denotes the space of $n$-dimensional column vectors
of nonnegative numbers. We also let $\Real^n$ be the space of $n$-dimensional column vectors 
$\Real^{d\times n}$ the space of  $d \times n$ real matrices, and $\SymMat^{ n}$ the space of  $n \times n$ symmetric matrices.

Let the number of the quadratic equalities of \eqref{PPO} is $m$,
the number of linear equality constraints $d$, the number of linear inequalities $e$. 
Then, with appropriately chosen  matrices $\Q_k  \in \SymMat^n \ (k=1,\ldots,m), \ \L_{m+1} \in \Real^{d \times n}, \L_{m+2} \in \Real^{e \times n},$ and
$\q_0 \in \Real^n$, we assume that \eqref{PPM} can be written in the following general form:
\begin{eqnarray*}
&\displaystyle \min_{\x \in \mathbb{R}^{n}}&  \q_0^T\x  \\  
&{\text{subject to }}&
 \x^{T} \Q_{k}   \x + \q_{k}^T \x + \gamma_{k}  = 0 \ \ \ (k=1, \dots, m), \ \ \\
&&  \ \L_{m+1} \x = \b_{m+1},  \  \ \L_{m+2} \x \leq \b_{m+2}, \\
&& \bell \leq \x \leq \u, \
\end{eqnarray*}
where $\bell, \u \in \Real^n$  mean, respectively, the lower and upper bounds for $\x$, and  $\b_{m+1} \in \Real^d, \b_{m+2} \in \Real^e$.

Consequently, \eqref{PPO} can be expressed as the following general form:
\begin{eqnarray}
&\displaystyle \min_{\x \in \mathbb{R}^{n}, \ \lambda \in \mathbb{R}^{m+d}}&   \q_0^T\x +\delta \sum_{i=1}^{m+d} \lambda_i \nonumber  \\
&{\text{subject to }}& 
-\lambda^1_k  \leq  \x^{T} \Q_k   \x +\q_{k}^T \x + \gamma_{k} \leq \lambda^1_k \ \ \ (k=1, \dots , m), \  \  \nonumber \\
 && -\blam^2_{m+1} \leq \L_{m+1} \x - \b_{m+1} \leq \blam^2_{m+1},   \  \  \L_{m+2}  \x \leq \b_{m+2}, \  \ \nonumber \\
&& \blam \geq 0 \ , \  \bell \leq \x \leq \u, \label{PPR}
\end{eqnarray}
where $\blam =[\blam^1,\blam^2_{m+1}]^T  \in \Real^{m+d}$ and $\blam^2_{m+1} \in \Real^d$.


If we let $\q_{m+r} $ be the $r$th row of  $\L_{m+1} \ (r=1,\ldots, d)$ and
$\q_{m+d+\rho}$ be the $\rho$th row of   $\L_{m+2} \ (\rho=1,\ldots, e)$, then
 \eqref{PPR} is written as follows:
\begin{eqnarray}
&\displaystyle \min_{\x \in \mathbb{R}^{n}, \ \blam \in \mathbb{R}^{m+d}}&   \q_0^T\x +\delta \sum_{i=1}^{m+d} \lambda_i \nonumber  \\
&{\text{subject to }}& 
 \x^{T} \Q_k   \x +\q_{k}^T \x - \lambda^1_k + \gamma_{k} \leq 0,   \nonumber \\
&&  -\x^{T} \Q_k   \x -\q_{k}^T \x - \lambda^1_k    -  \gamma_{k} \leq 0 \ (k=1, \dots , m), \  \  \nonumber \\
 &&  \q_{m+r} \x  - (\blam^2_{m+1})_r - (\b_{m+1})_r \leq 0, \nonumber \\
&&  - \q_{m+r} \x  -(\blam^2_{m+1})_r +  (\b_{m+1})_r \leq 0 \   (r=1,\ldots,d), \nonumber \\ 
 & &   \q_{m+r+\rho}  \x-  (\b_{m+2})_\rho \leq  0 \  (\rho=1,\ldots,e), \ \nonumber \\
&& \blam \geq 0 \ , \  \bell \leq \x \leq \u. \label{PPR1}
\end{eqnarray}
Notice that all the diagonal elements of $\Q_1, \ldots, \Q_m$ 
in the quadratic constraints  of  \eqref{PPR1} are zeros.
This  will be exploited in Section 4.

\section{SDP, SOCP and LP relaxations} \label{REX}

We give a brief description of an SDP relaxation of  general  QCQPs which include
 \eqref{PPO}. Then, SOCP \cite{KIM2003} and LP relaxations of general QCQPs are described using the scaled diagonally dominant (SDD) matrices and diagonally dominant (DD) matrices, respectively.

\subsection{SDP relaxations}

Let $\w \in \Real^n$.
Consider a general form of  QCQP:
\begin{eqnarray}
\begin{array}{rcl}
\zeta^* :=  \min & & \w^T \Q_0 \w  +  \bbeta_0^T \w + \gamma_0 \\
\mbox{subject to} & & \w^T \Q_k \w +  \bbeta_k^T \w + \gamma_k \le 0 \quad (k=1, \ldots, m),
\end{array}\label{eq:original problem}
\end{eqnarray}
where $\Q_k \in \SymMat^n, \ \bbeta_k \in \Real^n \ (k=0,\ldots,m)$ and $ \gamma_k \in \Real   \ (k=1, \ldots, m)$.
Since $\Q_k \in \SymMat^n \ (k=0,\ldots,m)$ is not necessarily positive semidefinite, \eqref{eq:original problem} 
is a nonconvex problem.

Introducing a new variable matrix $\W \in \SymMat^n$, we let
 \[ \bar{\Q}_k := \left(\begin{array}{cc} \gamma_k & \bbeta_k^T/2 \\ \bbeta_k/2 & \Q_k \end{array}\right), \
\bar{\W} := \left(\begin{array}{cc} w_{00} & \w^T \\ \w & \W \end{array}\right),
\mbox{  and  } \bar{\H}_0 := \left(\begin{array}{cc} 1 & \0^T \\ \0 & \O \end{array}\right). \]
Then, 
an SDP relaxation of (\ref{eq:original problem}) is given by
\begin{eqnarray}
\begin{array}{rcl}
\zeta_{SDP}^* := \min & & \bar{\Q}_0 \bullet  \bar{\W} \\
\mbox{subject to} & & \bar{\Q}_k \bullet \bar{\W} \le 0  \quad (k=1, \ldots, m), \\
&& \bar{\H}_0 \bullet \bar{\W} = 1, \\
& & \bar{\W} \in \SymMat_+^{n+1},
\end{array}\label{SDP}
\end{eqnarray}
where the inner product  $\bar{\Q} \bullet \bar{\W}$ means 
the standard inner product between two symmetric matrices, {\it i.e.,}
$\bar{\Q} \bullet \bar{\W} =\sum_{i}\sum_{k} \bar{Q}_{ik}\bar{W}_{ik}$.

\subsection{SOCP relaxations}
In  \cite{KIM2003}, an SOCP relaxation was proposed using the $2 \times 2$ principle submatrices of the variable matrix
$\bar{\W}$ of \eqref{SDP}. They showed that the SOCP relaxation provides the exact optimal solution
for QCQP if the off-diagonal elements of $\bar{\Q}_k$ $(k = 0,\ldots,m)$ are nopositive.  
The SOCP relaxation in \cite{KIM2003} is closely related to the dual of the first level relaxation of 
the hierarchy of the scaled diagonally dominant sum-of-squares (SDSOS) relaxations proposed
in \cite{AAA2014}.
By applying the approach in  \cite{KIM2003}, we obtain the following SOCP relaxation:
\begin{equation}
\left.
\begin{array}{llll} 
\mbox{min }   & \bar{\Q}_0 \bullet \bar{\W} \\
\mbox{subject to } & \bar{\Q}_k \bullet \bar{\W} \leq 0 \ (1 \leq k \leq m), \ \bar{\H}_{0}\bullet \bar{\W} = 1, \\ 
                   & 
                    \bar{W}_{jj} \geq 0 \ (1 \leq j \leq n+1), \\ 
                   & (\bar{W}_{ij})^2 \leq \bar{W}_{ii}\bar{W}_{jj} \ (1 \leq i < j \leq n+1).
\end{array}
\right\} 
\label{SOCPconstraints}
\end{equation}
Using
\begin{equation} 
           w^2 \leq \xi \eta,  \ \xi \geq 0 \ \mbox{ and } \eta \geq 0 \mbox{ if and only if }
          \left\Vert \left( \begin{array}{c} \xi -\eta \\ 2 w \end{array} \right)  \right\Vert \leq \xi +\eta,
\label{SOCEQ}
\end{equation} 
 (\ref{SOCPconstraints}) is converted to an SOCP.
Thus, the following SOCP  is equivalent to the 
problem (\ref{SOCPconstraints}). 
\begin{equation}
\left.
\begin{array}{llll} 
\zeta_{SOCP}^* := \mbox{min }   & \bar{\Q}_0 \bullet \bar{\W} \\
\mbox{subject to } &\bar{\Q}_p \bullet \bar{\W} \leq 0 \ (1 \leq p \leq m), \ \bar{\H}_{0}\bullet \bar{\W} = 1, \\
                   & \left\Vert \left(\begin{array}{c} \bar{W}_{ii} - \bar{W}_{jj} \\ 2 \bar{W}_{ij}
			 \end{array} \right) \right\Vert \leq \bar{W}_{ii} + \bar{W}_{jj} 
                     \ (1 \leq i < j \leq n+1).
\end{array}
\right\} 
\label{SOCP1FINAL}
\end{equation}
Since $\bar{\W} \in \SymMat^n_+$ implies $ (\bar{W}_{ij})^2 \leq \bar{W}_{ii}\bar{W}_{jj} \ (1 \leq i < j \leq n+1)$,
the optimal value $\zeta_{SOCP}^*$ of \eqref{SOCP1FINAL} is weaker than $\zeta_{SDP}^* $ of \eqref{SDP}:
\[ \zeta_{SOCP}^* \leq  \zeta_{SDP}^*  \leq \zeta^*.\]

\subsection{LP relaxations}
We derive LP relaxations of \eqref{eq:original problem} using 
 the diagonally dominant sum-of-squares relaxation (DSOS) in \cite{AAA2014}.

Consider the cone of diagonally dominant matrices of dimension $n+1$  defined by
\begin{eqnarray*}
\DC^{n+1} := \left\{\W \in \SymMat^{n+1} : W_{ii} \ge \sum_{j \neq i} |W_{ij}| \quad (1 \leq i \leq n+1)\right\}.
\end{eqnarray*}
In \cite{Barker75}, 
the dual of $\DC^{n+1}$ is given by
\begin{eqnarray*}
(\DC^{n+1})^* & :=& \left\{\W \in \SymMat^{n+1} : \w^T \W \w \ge 0 \mbox{ for } \forall \w \mbox{ with at most 2 nonzero elements
 } 1 \mbox{ or -1} \right\} \\
 & = &  \left\{\W \in \SymMat^{n+1} : 
                    W_{ii} \geq 0 \ (1 \leq i \leq n+1), \ 
                    W_{ii} +W_{jj} -2 |W_{ij}| \geq 0 \quad (1 \leq i < j \leq n+1) \right\} 
\end{eqnarray*}
Using $(\DC^{n+1})^*$, an LP relaxation of  \eqref{eq:original problem} can be derived as
\begin{equation}
\left.
\begin{array}{llll} 
\zeta_{LP}^* := \mbox{min }   & \bar{\Q}_0 \bullet \bar{\W} \\
\mbox{subject to } & \bar{\Q}_k \bullet \bar{\W} \leq 0 \ (1 \leq k \leq m), \ \bar{\H}_{0}\bullet \bar{\W} = 1, \\ 
                   & 
                    \bar{W}_{ii} \geq 0 \ (1 \leq i \leq n+1), \\ 
                   & \bar{W}_{ii} +\bar{W}_{jj} -2 |\bar{W}_{ij}| \geq 0 \quad (1 \leq i < j \leq n+1).
\end{array}
\right\} 
\label{LPR}
\end{equation}

Let $\bar{\W}$ be a feasible solution of \eqref{SOCP1FINAL}. Then,  $ |\bar{W}_{ij}| \leq \sqrt{\bar{W}_{ii} \bar{W}_{jj}} \  \ (1 \leq i < j \leq n+1). $ 
Since $\sqrt{\bar{W}_{ii} \bar{W}_{jj} } \le (\bar{W}_{ii} + \bar{W}_{jj})/2$ always holds for all
nonnegative $\bar{W}_{ii}$ and $\bar{W}_{jj} $,
  $\bar{\W}$ is a feasible solution of \eqref{LPR}. 
Thus, the LP relaxation \eqref{LPR} is an weaker relaxation than
the SOCP relaxation \eqref{SOCP1FINAL} and
 the following relation holds for the optimal values of the three relaxations:
\begin{equation} \zeta^*_{LP} \leq \zeta^*_{SOCP} \leq  \zeta^*_{SDP}  \leq \zeta^*. \label{EQUAL} \end{equation}


\section{The equivalence of the optimal values of SDP, SOCP and LP relaxations}

Now, we  show the equivalence among the optimal values of  \eqref{SDP}, \eqref{SOCP1FINAL} and \eqref{LPR}  under the
following assumptions. As mentioned at the end of  Section 2, the pooling problem satisfies the following assumption. 

\begin{assum}\label{diag-zero}
All the diagonal elements  in $\Q_0, \Q_1, \ldots, \Q_m$  of \eqref{eq:original problem} are zeros.
\end{assum}

\begin{theorem} \label{THEOREM}
Suppose that Assumption \ref{diag-zero} holds.
Then,
$\zeta_{SDP}^* = \zeta_{SOCP}^* = \zeta_{LP}^*$.
\end{theorem}

\begin{proof}
Let $\bar{\W}= \left(\begin{array}{cc} w_{00} & \w^T \\ \w & \W \end{array} \right)$ be a feasible solution of \eqref{LPR}. 
It always holds that  $w_{00} = 1$ by the constraint $\bar{\H}_0 \bullet \bar{\W} = 1$.
If we add a sufficiently large 
number $\alpha \ge \lambda_{\max}\left(\frac{\w \w^T}{w_{00}} - \W \right)$
 to the diagonal of $\bar{\W}$ except the first diagonal element $w_{00}$ of $\bar{\W}$,
the resulting 
matrix $\left(\begin{array}{cc}  w_{00}  & \w^T \\ \w & \W +\alpha \I \end{array}\right)$
becomes positive semidefinite  by the Schur complement. Here $\lambda_{\max}$ means the largest eigenvalue.
 The inequality constraints, however, still hold and 
the objective value remains same, 
since the diagonal elements in $\Q_0, \ldots, \Q_m$ are zeros by Assumption \ref{diag-zero}. Thus,  
$\left(\begin{array}{cc} w_{00}  & \w^T \\ \w & \W +\alpha \I \end{array}\right)$
is a feasible solution of the SDP relaxation
\eqref{SDP}. 
Therefore, we can construct a feasible solution in the SDP relaxation whose objective value
is same as $\bar{\W}$, and this leads to $\zeta_{SDP}^* \leq \zeta_{LP}^*$.
In view of this with  \eqref{EQUAL},
the desired result $\zeta_{SDP}^* = \zeta_{SOCP}^* = \zeta_{LP}^*$  follows.
\end{proof}

From Theorem~\ref{THEOREM}, 
we show the relationship among the optimal values of the primal and dual problems in the subsequent discussion.
The dual of (\ref{SDP}) can be written as
\begin{eqnarray}
\begin{array}{rcl}
\mu_{SDP}^*: = \max && \mu  \\
\mbox{subject to} && \bar{\Q}_0 + \sum_{k=1}^m \eta_k \bar{\Q}_k 
- \mu \bar{\H}_0 - \bar{\S} = \O, \\
& & \eta_1, \ldots, \eta_m \ge 0,\  \mu \in \Real, \\
& &  \bar{\S} \in \SymMat_+^{n+1}.
\end{array}\label{SDPD}
\end{eqnarray}
By Assumption~\ref{diag-zero}, this problem has no interior point.
Thus, the positive duality gap 
between $\zeta_{SDP}^*$ and $\mu_{SDP}^*$ might exist as  the Slater condition does not hold.
In the following Corollary~\ref{COROL}, we show  that there is no duality gap, that is, 
$\zeta_{SDP}^* = \mu_{SDP}^*$, using the dual of 
the SOCP relaxation \eqref{SOCP1FINAL} and
the LP relaxation \eqref{LPR}.
These dual problems are closely related to scaled diagonally dominant sum of squares (SDSOS) and diagonally dominant
sum of squares (DSOS) in \cite{AAA2014}.

In \cite{AAA17}, SDSOS relaxations were proposed using SDD matrices.
A matrix $\B \in \SymMat^{n}$ is
SDD if and only if it can be expressed as
\[ \B = \sum_{i =1}^{n-1} \sum_{ j = i+1}^n \B^{ij}, \]
where the nonzero elements of $ \B^{ij} \in \SymMat^{n}$ are from the $2 \times 2$ principal submatrix of a positive semidefinite matrix
$\C \in \S^n_+$ with
$i$th and $j$th rows and columns of $\C$ and all the other
elements of $ \B^{ij}$ are zero. 
More precisely, $\B^{ij}$ is a symmetric matrix  with nonzero elements 
only in $(i,i)$th, $(i,j)$th, $(j,i)$th and $(j,j)$th positions such that
 $\left[ \begin{array}{cc} (\B^{ij})_{ii} & (\B^{ij})_{ij}  \\ (\B^{ij})_{ij}  & (\B^{ij})_{jj}  \end{array} \right] \in \SymMat^{2}_+$.
 Thus, each $\B^{ij} \in \SymMat^n$ is positive semidefinite. 
Let $\SC \DC^{n}$ be the cone of SDD matrices.
It is well-known that any DD matrix is SDD,
therefore, 
$\DC^{n} \subset \SC \DC^{n} \subset \SymMat_+^{n}$ holds.

Replacing $\bar{\S} \in \SymMat_+^{n+1}$ in \eqref{SDPD} by $\bar{\S} \in \SC \DC^{n+1}$ corresponds to
the first level of the hierarchy of SDSOS relaxation for QCQPs  in \cite{AAA17}, and it is the dual of
 \eqref{SOCP1FINAL}: 
\begin{eqnarray}
\begin{array}{rcl}
\mu_{SOCP}^*: = \max && \mu  \\
\mbox{subject to} && \bar{\Q}_0 + \sum_{k=1}^m \eta_k \bar{\Q}_k 
- \mu \bar{\H}_0 - \bar{\S} = \O, \\
& & \eta_1, \ldots, \eta_m \ge 0,\  \mu \in \Real, \\ 
& & \bar{\S} \in \SC\DC^{n+1}. \ \end{array} \label{SOCPD}
\end{eqnarray}

The dual of \eqref{LPR} is given as
\begin{eqnarray}
\begin{array}{rcl}
\mu_{LP}^*: =  \max && \mu  \\
\mbox{subject to} & & \bar{\Q}_0 + \sum_{k=1}^m \eta_k \bar{\Q}_k 
- \mu \bar{\H}_0 - \bar{\S} = \O, \\
& & \eta_1, \ldots, \eta_m \ge 0, \ \mu \in \Real,\\
& &  \bar{\S} \in \DC^{n+1}.
\end{array}\label{LPD}
\end{eqnarray}
In general,
\begin{equation} \mu_{SDP}^* \geq \mu_{SOCP}^* \geq \mu_{LP}^* \label{DEQUAL} \end{equation}
holds from $\DC^{n+1} \subset \SC \DC^{n+1} \subset \SymMat_+^{n+1}$.

We will show that the primal problems and the dual problems attain the same optimal values, and this indicates that there is no duality gap 
between the SDP relaxation \eqref{SDP} and \eqref{SDPD}.


\coro \label{COROL}
Under Assumption \ref{diag-zero}, 
it holds that 
\begin{eqnarray*}
\zeta^*_{SDP} = \zeta^*_{SOCP} = \zeta^*_{LP} = \mu_{LP}^* = \mu_{SOCP}^* = \mu_{SDP}^*.
\end{eqnarray*}
\ecoro
\begin{proof} 
For \eqref{SDPD},
we let 
$\bar{\S} = \left(\begin{array}{cc} s_{00} & \s^T \\ \s & \S \end{array}\right)$.
From Assumption \ref{diag-zero}, the diagonal of $\S$ is zero.
Since $\S \in \SymMat^n_+$,  we have $\S=\O$, thus,  $\bar{\S} \in \SymMat^{n+1}_+$ leads to $\s = \0$.
Hence, $\eqref{SDPD}$ is equivalent to the following problem:
\begin{eqnarray}
\begin{array}{rcl}
\max & & \mu  \\
\mbox{subject to} & & \gamma_0 + \sum_{k=1}^m \eta_k \gamma_k - \mu - s_{00} \ge 0, \\
& & \q_0 + \sum_{k=1}^m \eta_k \q_k = \0, \\
& & \Q_0 + \sum_{k=1}^m \eta_k \Q_k = \O, \\
& & \eta_1, \ldots, \eta_m \ge 0, \ \mu \in \Real, \ s_{00} \ge 0.
\end{array} \label{SDPD2}
\end{eqnarray}
Similarly, for \eqref{LPD}, we can show that $\S = \O$ and $\s = \0$  using the zero diagonal of $\bar{\S}$.
As a result, \eqref{LPD} is equivalent to \eqref{SDPD2},
and the optimal values of  \eqref{LPD} and \eqref{SDPD2} coincide, i.e., $\mu_{SDP}^*= \mu_{LP}^*$. 
Since the duality theorem holds on linear programming problems regardless of the existence of interior points, 
the optimal values of \eqref{LPR} and \eqref{LPD} are equivalent, {\it i.e.,} $\zeta_{LP}^* = \mu_{LP}^*$.
%
By $\mu_{SDP}^*= \mu_{LP}^* = \zeta_{LP}^*$, \eqref{EQUAL},
\eqref{DEQUAL} and Theorem~\ref{THEOREM},
the desired result follows. 


 \end{proof}

\section{Computational methods}

In this section, we discuss two computational methods  for adjusting and refining an approximate solution obtained by
the SOCP or LP relaxation of the pooling problem.
 As the pooling problem is NP-hard, only approximate solutions
 can be obtained by the relaxation methods. In addition,
 the bounds for the variables of the pooling problem have been modified when 
 binary variables have been removed in
 \eqref{RELAXBOUND}. 
 As a result,  an approximate solution by SOCP or LP relaxation may not be a solution to the original problem.  
 
 In Nishi's method  \cite{NISHI10},
 a mixed-integer linear program was first solved for finding a feasible solution of the original pooling problem.
 Then, 
 {\tt fmincon} in Matlab,  a nonlinear programming solver, was applied to find a local optimum solution. This step
 turned out to be very time-consuming. To improve the computational efficiency for finding a  solution that satisfies the plant requirements,
 we propose a rescheduling method based on  successive refining the solution obtained by solving 
the SOCP or LP problem.

 Nishi \cite{NISHI10}'s method can be described as follows:  \\
 \medskip
\centerline{\fbox{\parbox{\textwidth}{
{\bf Algorithm 5.1}: {\bf Nishi's method}
\begin{description}
\item[Step 1.] Solve an SDP relaxation  of the pooling problem \eqref{PPM}.
\item[Step 2.] Apply a procedure called FFS (finding feasible solution)
 to find a feasible solution $\x$. 
\item[Step 3.] Use a general nonlinear programming solver ({\tt fmincon}) starting from $\x$ for a local optimal solution.
\end{description}
}}}
We note that the feasible solution $\x$ obtained in Step 2 is not necessarily a local minimum of the original problem.
Step 3  is very time-consuming, as shown in numerical results in Section \ref{NUMER}.

In our method,  the SOCP or LP relaxation is used instead of the SDP relaxation.
We also propose a rescheduling method for Step 3 of Nishi's method.
More precisely, after applying applying the SOCP or LP relaxation and FFS,  which is described in Section \ref{FFS} in detail,
an approximate solution is further refined by the rescheduling method.
The main steps of our  method
is described as follows: \\[5pt]
\centerline{\fbox{\parbox{\textwidth}{
{\bf Algorithm 5.2}: {\bf The proposed method}
 \begin{description}
\item[Step 1.] Solve the SOCP or  LP relaxation \eqref{LPR} of the pooling problem.
\item[Step 2.] Apply  FFS 
to obtain a feasible solution $\x$. 
\item[Step 3.] Perform the proposed rescheduling method. 
\end{description}
}}}
We briefly review FFS  \cite{NISHI10} in Section  \ref{FFS}  and describe our proposed rescheduling method in Section \ref{RES}.

\subsection{A method for finding a feasible solution} \label{FFS}

As a solution attained by the SDP, SOCP or LP relaxation is not necessarily feasible for the original problem,
the following mixed-integer linear problem  was introduced to find a feasible solution in \cite{NISHI10} as the first step of the procedure  FFS. More precisely,
 the solution $(\bar{\p}, \bar{\q})$ obtained by the relaxation methods is used for the following problem called FFS1:
\begin{eqnarray}
\text{FFS1
 $(\bar{\p}, \bar{\q})$:=} &\displaystyle \min_{\a,\p,\q,\v,\u,\s}& \ \alpha \sum_{t \in T} \sum_{i \in V} s_{i}^{t} 
+ \sum_{t \in T} \sum_{(i, j) \in A} CA_{ij}a_{ij}^{t} +\sum_{t \in T} \sum_{i \in V_{P}} v_{i}^{t} \nonumber \\
&\text{s.t. }& -\s \leq \p - \bar{\p} \leq \s, \nonumber \\ 
&& p^{t+1}_{i} = p^{t}_{i} + SA^{t}_{i} - \sum_{k \in E(i)} a_{ik}^{t}, \  \ p^{t}_{i} \geq 0
 \ (i \in V_{S}, t \in T), \nonumber \\
 && p^{t+1}_{i} = p^{t}_{i} + \sum_{j \in I(i)} a_{ji}^{t} - \sum_{k \in E(i)} a_{ik}^{t} \ (i \in V_I, t \in T), \nonumber \\
&&p_{i}^{\min} \leq p_{i}^{t} \leq p_{i}^{\max} \ (i \in V_{I}, t \in T), \nonumber \\
&& RC_{i}^{t} = \sum_{j \in I(i)} a_{ji}^{t} \ , \ RC_{i}^{t}q_{i}^{t} = \sum_{j \in I(i)} a_{ji}^{t}\bar{q}^{t}_{j}
 (i \in V_{P}, t \in T), \nonumber \\
&&q_{i}^{t} \geq RQ_{i}^{t} - v_{i}^{t} \ , \ v_{i}^{t} \geq 0 \  (i \in V_{P}, t \in T), \nonumber \\
&& u^{t}_{ij} L_{ij} \leq a_{ij}^{t} \leq u^{t}_{ij}U_{ij} \ , \ u_{ij}^{t} \in \{ 0, 1\} \ ((i, j) \in A , t \in T), \nonumber \\
&& \sum_{j \in I(i)} u_{ji}^{t} + \sum_{k \in E(i)} u_{ik}^{t}  \leq 1 \ (i \in V, t \in T), \label{FFSPR}
\end{eqnarray}
where $\alpha$ denotes a weight coefficient for  $\parallel \p - \bar{\p} \parallel_{1}$ in the objective function.
We note that \eqref{FFSPR} is a mixed-integer linear programming problem, 
thus it is computationally efficient to solve
\eqref{FFSPR}. 

After solving \eqref{FFSPR}, a procedure FFS2
 to further refine a feasible solution of the pooling problem is employed 
using the solution of  \eqref{FFSPR} in  \cite{NISHI10}. More precisely,
using the output  $(\hat{\a}, \hat{\u}, \hat{\p}, \hat{\q}, \hat{\v})$ of \eqref{FFSPR},
FFS2 produces $\tilde{\q}$ and $\tilde{\v}$ 
by
\begin{eqnarray}
&&\left\{ \begin{array}{l}
\tilde{q}_{i}^{1} = q_{i}^{1}\nonumber \\
\tilde{q}_{i}^{t+1} = (\hat{p}_{i}^{t}\hat{q}_{i}^{t} + SA_{i}^{t}SQ_{i}^{t} -
\sum_{k \in E(i)} \hat{a}_{ik}^{t}\tilde{q}_{i}^{t})/\hat{p}_{i}^{t+1}\ \ \ (i \in V_{S}, t \in T)
\end{array}\right. \nonumber \\
&&\left\{ \begin{array}{l}
\tilde{q}_{i}^{1} = q_{i}^{1} \nonumber \\
\tilde{q}_{i}^{t+1} = (\hat{p}_{i}^{t}\hat{q}_{i}^{t} + \sum_{j \in E(i)} \hat{a}_{ji}^{t}\tilde{q}_{j}^{t}
- \sum_{k \in E(i)} \hat{a}_{ik}^{t}\tilde{q}_{i}^{t})/\hat{p}_{i}^{t+1} \ \ \ (i \in V_I, t \in T)
\end{array}\right. \nonumber \\
&&\left\{ \begin{array}{l}
\tilde{q}_{i}^{t} = (\sum_{j \in I(i)} \hat{a}_{ji}^{t}\tilde{q}_{j}^{t})/RC_{i}^{t} \ \ \  (i \in V_{P}, t \in T) \\
\tilde{v}_{i}^{t} = \max \{ 0, RQ_{i}^{t} - \tilde{q}_{i}^{t}\}.
\end{array}\right. \label{REFINE}
\end{eqnarray}
Consequently, a solution
$(\hat{\a}, \hat{\p}, \tilde{\q}, \hat{\u}, \tilde{\v})$ of the pooling problem is attained. 
Notice that
solving  \eqref{FFSPR} determines $(\hat{\a}, \hat{\u}, \hat{\p})$ using $\bar{\p}, \bar{\q}$,
and   a feasible solution $(\tilde{\q}, \tilde{\v})$ is obtained  by \eqref{REFINE} using  $(\hat{\a}, \hat{\u}, \hat{\p})$.
We also note that the equations in \eqref{REFINE} involves nonlinear terms.
It should be mentioned that 
a feasible solution $(\hat{\a}, \hat{\p}, \tilde{\q}, \hat{\u}, \tilde{\v})$ 
for the pooling problem does not necessarily satisfy 
the quality requirement $\tilde{v}_i^t = 0$. 
For this requirement, we refine the solution by a rescheduling method.

\subsection{A rescheduling method} \label{RES}

We propose a rescheduling method to refine the obtained solution $(\hat{\a}, \hat{\p}, \tilde{\q}, \hat{\u}, \tilde{\v})$ from the procedure
in Section \ref{FFS}. 
 The rescheduling method is based on successive refinement of the solution. The algorithm continues until  all requirements are satisfied, {\it i.e.,}
 $v_i^{t}=0$ $(i \in V_P, t \in T)$, 
or it is determined that the successive refinement cannot satisfy the plant requirements
in Step 3.6.

We denote the starting time step for the rescheduling method as  $\hat{t}$. 
The rescheduling algorithm   for Step 3 of Algorithm 5.2 is described as follows: \\[5pt]
\centerline{\fbox{\parbox{\textwidth}{
{\bf Algorithm 5.3}: {\bf The proposed rescheduling method}
\begin{description}
\item[Step 3.1.] Initialize $\hat{t}=0$. Set  $\x^*$ as the zero vector 
 of dimension $n$, the length of $(\a, \p, \q, \u, \v)$.
\item[Step 3.2.] Formulate the pooling problem with discretized time step $\hat{T} = \{ \hat{t}+1, \dots , M_{T} \}$.
\item[Step 3.3.] Solve    SOCP  \eqref{SOCP1FINAL} (or LP \eqref{LPR})   relaxation formulated for $\hat{T}$, apply FFS \eqref{FFSPR} and \eqref{REFINE}
to obtain a feasible solution $\x^{+}= (\hat{\a}, \hat{\p}, \tilde{\q}, \hat{\u}, \tilde{\v})$.
\item[Step 3.4.] If  $\x^{+}$ satisfies the requirement $\tilde{v}_i^{t}=0$ for all  $(i, t) \in V_P \times \hat{T}$, then replace $\x^*$ 
with $\x^{+}$ for $t \in \hat{T}$, output $\x^*$
and terminate. 
\item[Step 3.5.] Find the smallest time step $t^+$ such that $\tilde{v}_{i}^{t^+} > 0$ for some $i \in V_{P}$.
\item[Step 3.6.] Modify $\x^{+}$ as $\breve{\x} = (\breve{\a}, \breve{\p}, \breve{\q}, \hat{\u}, \breve{\v})$
and replace $\x^*$ with $\breve{\x}$ 
for the time steps $\{ \hat{t}+1, \dots , t^{+} \}$.
\item[Step 3.7.] Let $\hat{t} = t^{+}$. Return to Step 3.2. (If $\hat{t} = M_{T}$, output $\x^*$ and stop.)
\end{description}
}}}
Note that the size of the relaxation problem 
solved in Step~3.3 will become smaller  as $\hat{t}$ approaches to $M_T$.
Steps 3.2 and 3.3 can be skipped at $\hat{t}=0$ as Steps 1 and 2 of Algorithm 5.2 have been performed.

Step 3.6 plays an important role for the overall performance of the rescheduling method, in particular, to successfully find a solution to the pooling problem.
At Step 3.6, $\x^{+} = (\hat{\a}, \hat{\p}, \tilde{\q}, \hat{\u}, \tilde{\v})$ is modified for the time steps $\{ \hat{t}+1, \dots , t^{+} \} $ 
 to obtain a solution $\breve{\x}= 
(\breve{\a}, \breve{\p}, \breve{\q}, \breve{\u}, \breve{\v})$ 
which satisfies the plant demand in the time steps $\{ \hat{t}+1, \dots , t^{+} \}$. 

More precisely, in Step~3.6, 
for each $t \in \{ \hat{t}+1, \dots , t^{+} \}$, 
 we first check whether $\tilde{q}_{i}^{t} \  (i \in V_{P})$ satisfies the plant requirements, {\it i.e.}, $\tilde{v}_i^t =0$ or not.
Depending on the computed values of $\tilde{q}^{t}_{i} \ (i \in V_{P})$, we consider two cases:
(Case I) if $\tilde{q}^{t}_{i} \geq RQ^{t}_{i}$ holds for all $i \in V_{P}$, it means that  we have excessive supplies,
(Case II) if there exists some $i \in V_{P}$ that satisfies $\tilde{q}^{t}_{i} < RQ^{t}_{i}$, then 
it means  shortage in supplies. 

For (Case I), we modify the requirement $\breve{q}_i^t$ by $RQ_i^t$ at  plants $i \in V_P$ and
time $t$, 
to reduce excessive supplies so that   
 more quantities can be available  at the intermediate tanks   
 for the subsequent 
 modifications in ${\hat{t}  +1 ,\ldots, t^+}$. 
This modification on  the requirement at plant $i \in V_P$ in turn
affects the intermediate tanks $j \in V_I$ along with the network arcs  determined by
$\tilde{u}_{ij}^t = 1$ in FFS1, and 
to the source $k \in V_S$.
The algorithm for (Case I) is described as Algorithm~5.4(I), which 
is employed as Step~3.6 in Algorithm~5.3.\\[5pt]
\centerline{\fbox{\parbox{\textwidth}{
{\bf Algorithm 5.4(I) [The case of excessive supplies]}: 
\newline
For the time step $\hat{t}$, 
let $\breve{p}_i^{\hat{t}} = \hat{p}_i^{\hat{t}}$ and
$\breve{q}_i^{\hat{t}} = \tilde{q}_i^{\hat{t}} \ (i \in V_{S} \cup V_{I})$.
\newline
For each $t = \hat{t}+1,\ldots, t^+$, apply the following steps:
\begin{description}
\item[Step 3.6.(I)1. ] For $i \in V_{P}$, set $\breve{q}^{t}_{i}=RQ^{t}_{i}$ and $\breve{v}^{t}_{i}=0$.
\item[Step 3.6.(I)2. ]
For $j \in V_I$ and $i \in V_{P}$ such that $(j,i) \in A$,
if $\tilde{u}_{ji}^t = 1$, 
compute $\breve{a}^{t}_{ji} = \hat{a}^{t}_{ji} - RC^{t}_{i}(\tilde{q}^{t}_{i} - RQ^{t}_{i})/\breve{q}^{t}_{j}  \                           
$,
otherwise 
set $\breve{a}^{t}_{ji} = \hat{a}^{t}_{ji}$.
\item[Step 3.6.(I)3. ]
For $i \in V_{P}$, compute   
$\breve{p}^{t+1}_{i} = \hat{p}^{t+1}_{i}$ and
$ \breve{q}^{t+1}_{i} = \frac{1}{RC^{t}_{i}} \sum_{j \in I(i)} \breve{a}^{t}_{ji}\breve{q}^{t}_{j}.$
\item[Step 3.6.(I)4. ] For $j \in V_I$, compute
$\breve{p}^{t+1}_{j} = \breve{p}^{t}_{j} - \breve{a}^{t}_{ji}$ and  $ \breve{q}^{t+1}_{j} = \breve{q}^{t}_{j}$. 
\item[Step 3.6.(I)5. ] 
For $k \in V_S$, set $\breve{p}_k^{t+1} = \breve{p}_k^{t}$ and  
 $\breve{q}_k^{t+1} = \breve{q}_k^{t}$.
\item[Step 3.6.(I)6. ] 
For $k \in V_S$ and $j \in V_I$ such that $(k,j) \in A$,
if $\tilde{u}_{kj}^t = 1$,
compute
$\breve{a}^{t}_{kj} = \min \{ U_{kj},\ \breve{p}^{t}_{k}, \ p^{\max}_{j} - \breve{p}^{t}_{j}  \}$, 
and adjust 
 $\breve{p}^{t+1}_{k}, \ \breve{q}^{t+1}_{k},\
\breve{p}^{t+1}_{j}$ and $\breve{q}^{t+1}_{j}$  
by
\begin{eqnarray*}
&&\hspace*{-30pt}\breve{p}^{t+1}_{k} = \breve{p}^{t}_{k} - \breve{a}_{kj}^{t},\ \breve{q}^{t+1}_{k} = \breve{q}^{t}_{k} 
   \ 
 \mbox{(if $\breve{p}^{t+1}_{k}=0$, then set $\breve{q}^{t+1}_{k} =0),$}  \label{eq4}\\
&&\hspace*{-30pt}\breve{p}^{t+1}_{j} = \breve{p}^{t}_{j} + \breve{a}_{kj}^{t},\ 
\breve{q}^{t+1}_{j} = (\breve{p}^{t}_{j}\breve{q}^{t}_{j} + \breve{a}^{t}_{kj}\breve{q}^{t}_{k})/\breve{p}^{t+1}_{j}. \label{eq5}
\end{eqnarray*}
\item[Step 3.6.(I)7. ] 
For $(i,j) \in A$ such that $\tilde{u}_{ij}^t = 0$,
set $\breve{a}_{ij}^t = 0$.
\item[Step 3.6.(I)8. ]
 For $i \in V_{P}$, recalculate
$\breve{q}^{t+1}_{i} = \frac{1}{RC^{t}_{i}} \sum_{j \in I(i)} \breve{a}^{t}_{ji}\breve{q}^{t}_{j}. $
\end{description}
}}}

For (Case II), 
Algorithm~5.4(II)  is applied. 
The output of the FFS procedure is $\hat{u}_{ij}^t$. 
If $\hat{u}_{ij}^t = 1$, then it means that the arc $(i,j) \in A$
should be connected at time $t$.
With the connected arcs constructed from the output  $\hat{u}_{ij}^t$, it cannot be guaranteed 
that the quality requirements at the plants are satisfied.
Thus, we perform the following steps to meet the quality requirements:
Between the intermediate tanks $V_I$ and the plants $V_P$, we first consider 
the arcs
(denoted in $A_{IP}$ in Algorithm~5.4(II))
that need the greatest requirements at the plants.
Then,  between the  sources $V_S$ and the intermediate tanks $V_I$,  the arcs 
(denoted in $A_{SI}$ in Algorithm~5.4(II))
that provide more quantities to the intermediate tanks from the sources are used.
\\[5pt]
\centerline{\fbox{\parbox{\textwidth}{
{\bf Algorithm 5.4(II)}: {\bf [The case of insufficient supplies]}
\newline
For the time step $\hat{t}$, 
let $\breve{p}_i^{\hat{t}} = \hat{p}_i^{\hat{t}}$ and
$\breve{q}_i^{\hat{t}} = \tilde{q}_i^{\hat{t}} \ (i \in V)$.
\newline
For each $t = \hat{t}+1,\ldots, t^+$, apply the following steps:
\begin{description}
\item[Step 3.6.(II)1. ] 
For each plant node $i \in V_{P}$, compute the requirements  $D_i^t := RC^{t}_{i} \times RQ^{t}_{i}$.
For each intermediate tank $j \in V_I$, calculate the maximum supply value 
 defined by $S_j := \min \{ \breve{p}^{t}_{j} - p^{\min}_j,\ U_{ji} \} \times \breve{q}^{t}_{j}$.
Sort the plants and intermediate tanks in the descending order 
such that $D_{M_S+M_I+1}^t \ge \ldots \ge D_{M_S+M_I+M_P}^t$ and 
 $S_{M_S+1}^t \ge \ldots \ge S_{M_S+M_I}^t$.
\item[Step 3.6.(II)2. ] 
For each $i \in V_P$, find $j_i \in V_I$ such that $(j_i, i) \in A$, 
$S_{j_i}^t \ge D_i^t$ and $j_i > j_{i-1} >  \cdots > j_{M_S+M_I+1}$.
We  denote the set of such matching arcs by  $A_{IP} := \{(j_i, i) \in A : i=M_S+M_I+1, \ldots, M_S+M_I+M_P$\}.
If such matching arcs cannot be found, return $\x^{+}$.
\item[Step 3.6.(II)3. ]
For $i \in V_P$, set $\breve{v}^{t}_{i}=0$. For $(j,i) \in A_{IP}$,
calculate $\breve{q}^{t}_{i},\ \breve{a}^{t}_{ji},\ \breve{p}^{t+1}_{j}$ and $\breve{q}^{t+1}_{j}$  by
$\breve{q}^{t}_{i}=RQ^{t}_{i}, \  
  \breve{a}^{t}_{ji}= \frac{1}{\breve{q}^{t}_{j}}RC^{t}_{i}RQ^{t}_{i}\label{eq14}, \
  \breve{p}^{t+1}_{j} = \breve{p}^{t}_{j} - \breve{a}^{t}_{ji},$ and $\breve{q}^{t+1}_{j}=\breve{q}^{t}_{j}.
$ 
\item[Step 3.6.(II)4. ] For each source $k \in V_{S}$, compute the maximum supply value of 
$\bar{S}_k^t := \min \{  \breve{p}^{t}_{k},\ U_{kj}\} \times \breve{q}^{t}_{k}$.
Sort the sources in the descending order such that $\bar{S}_1^t \ge \ldots \ge \bar{S}_{V_S}^t$.
Let $J := V_I \backslash \{j \in V_I : (j, i) \in A_{IP} \ \mbox{for some} \ i \in V_P\}$ and
$\bar{M} := \min \{ M_{S}, (M_I - M_{P})\}$.
Find an arc set $A_{SI} := \{(k_\alpha, j_\alpha) \in (V_S \times J) \cap A | \alpha = 1, \ldots, \bar{M} \}$  
such that $\bar{S}_{k_\alpha} \ge \bar{S}_{j_\beta}$
and $\bar{S}_{j_\alpha} \le \bar{S}_{j_\beta}$ for $\alpha < \beta$.
\item[Step 3.6.(II)5. ] For $(k,j) \in A_{SI}$, compute $\breve{a}^{t}_{kj}, \breve{p}^{t+1}_{k}, \breve{q}^{t+1}_{k},  \breve{p}^{t+1}_{j}$ and $\breve{q}^{t+1}_{j}$ by
\begin{eqnarray*}
&&\hspace*{-30pt}\breve{a}^{t}_{kj} = \min \{ \breve{p}^{t}_{k},\ U_{kj},\ p^{\max}_{j} -\breve{p}^{t}_{j} \}, \ 
 \breve{p}^{t+1}_{k} = \breve{p}^{t}_{k} - \breve{a}^{t}_{kj},\
  \breve{q}^{t+1}_{k} = \breve{q}^{t}_{k}, \label{eq17}\\ 
 &&\hspace*{-30pt} \breve{p}^{t+1}_{j} = \breve{p}^{t}_{j} + \breve{a}^{t}_{kj} ,\ 
  \breve{q}^{t+1}_{j} = (\breve{p}^{t}_{j}\breve{q}^{t}_{j}  + \breve{a}^{t}_{kj}\breve{q}^{t}_{k})/\breve{p}^{t+1}_{j}.\label{eq18}
  \end{eqnarray*}
\item[Step 3.6.(II)6. ] For each arc $(i, j) \in A$ whose two nodes ($i$ and $j$) have not been updated in the previous steps, set $\breve{a}^{t}_{ij}=0$ to indicate that the arc $(i,j)$ is unused at time $t$. Compute 
$\breve{p}^{t+1}_{i} = \breve{p}^{t}_{i}, \ \breve{q}^{t+1}_{i} = \breve{q}^{t}_{i}, \  
\breve{p}^{t+1}_{j} = \breve{p}^{t}_{j}$ and $\breve{q}^{t+1}_{j} = \breve{q}^{t}_{j}$. 
\item[Step 3.6.(II)7.]
For $i \in V_{P}$, recalculate
$\breve{q}^{t+1}_{i} = \frac{1}{RC^{t}_{i}} \sum_{j \in I(i)} \breve{a}^{t}_{ji}\breve{q}^{t}_{j}. \label{eq0}$
\end{description}
}}}
Note that the proposed rescheduling method  is a heuristic method,
therefore, there still remains  possibility that the rescheduling method 
cannot meet all quality requirements.
In that case, the rescheduling method terminates with $\x^+$ at Step 3.6.(II).2. 

\section{Numerical results} \label{NUMER}
The main purposes of our numerical experiments are to see whether  the optimal values
of the SDP, SOCP and LP relaxation coincide as shown in Theorem \ref{THEOREM},
and to demonstrate the numerical efficiency of the rescheduling method over 
  Matlab function {\tt fmincon} used in Nishi's  method \cite{NISHI10}. 
We also illustrate the computational efficiency of the proposed method by solving large-sized problems  from 
the standard pooling test problems.

For numerical experiments,  we first test the eight instances in Nishi \cite{NISHI10} to compare our results with those in \cite{NISHI10}.
The eight test instances in \cite{NISHI10} which have a solution satisfying the plant requirements are illustrated in Figures \ref{F-1}
and \ref{F-2}, 
and their sizes are shown in Table \ref{T-3}.
 For example, instance 8 has 2 sources, 4 intermediate tanks and 2 plants with time discretization
28, which has the same number of nodes as an instance with $8 \times 28$ nodes without time discretization.
The intermediate tanks in the test instances are connected to each other as the complete graph. 
Second, 
we generated two larger test instances, instance 9 and 10, whose sizes are shown
in Table  \ref{T-3}. 
More precisely, the numbers of sources, intermediate tanks and plants of the test instances are increased up to  20, 10, respectively, with 
time discretization $M_T=2$, to see the computational efficiency of the SOCP and LP relaxation.
The pipelines between the intermediate tanks of the instances 9 and 10 are also
connected as the complete graph.  
 The number of variables of the
two instances are 1344 and 1616, respectively.
Third, we tested on Foulds 3, Foulds 4 and Foulds 5   whose sizes are larger than the other  standard pooling test problems
 \cite{ALFAKI13}.

Numerical experiments were conducted on a Mac with OS X EI Captin version $10.11$, processor $3.2$GHz Intel Core i$5$, 
memory $8$GB, $1867$MHz DDR$3$, MATLAB\_R$2016$a.

To implement Nishi's  method \cite{NISHI10} based on Algorithm 5.1, 
SparsePOP \cite{WAKI2008}, a general polynomial optimization problem solver that includes {\tt fmincon} after solving the SDP relaxation, 
was applied to the test instances.
When {\tt fmincon} was used as a nonlinear programming solver for Nishi's method,  parameters were changed with 
 TolFun=$10^{-1}$ and  TolCon=$10^{-3}$.
For our proposed method described in Algorithm 5.2, 
we used SPOTless \cite{SPOTLESS} and MOSEK \cite{MOSEK} to solve the SOCP \eqref{SOCP1FINAL}  and LP relaxation \eqref{LPR} and CPLEX \cite{CPLEX} to solve FFS \eqref{FFSPR}, and applied our rescheduling method Algorithm 5.3.

We used $10^{-4}$ for 
the penalty weight $\delta$  in our proposed formulation \eqref{PPR}.

\begin{figure}[hbt]
\hspace*{-1.7cm}
\includegraphics[width=11cm, height=7cm]{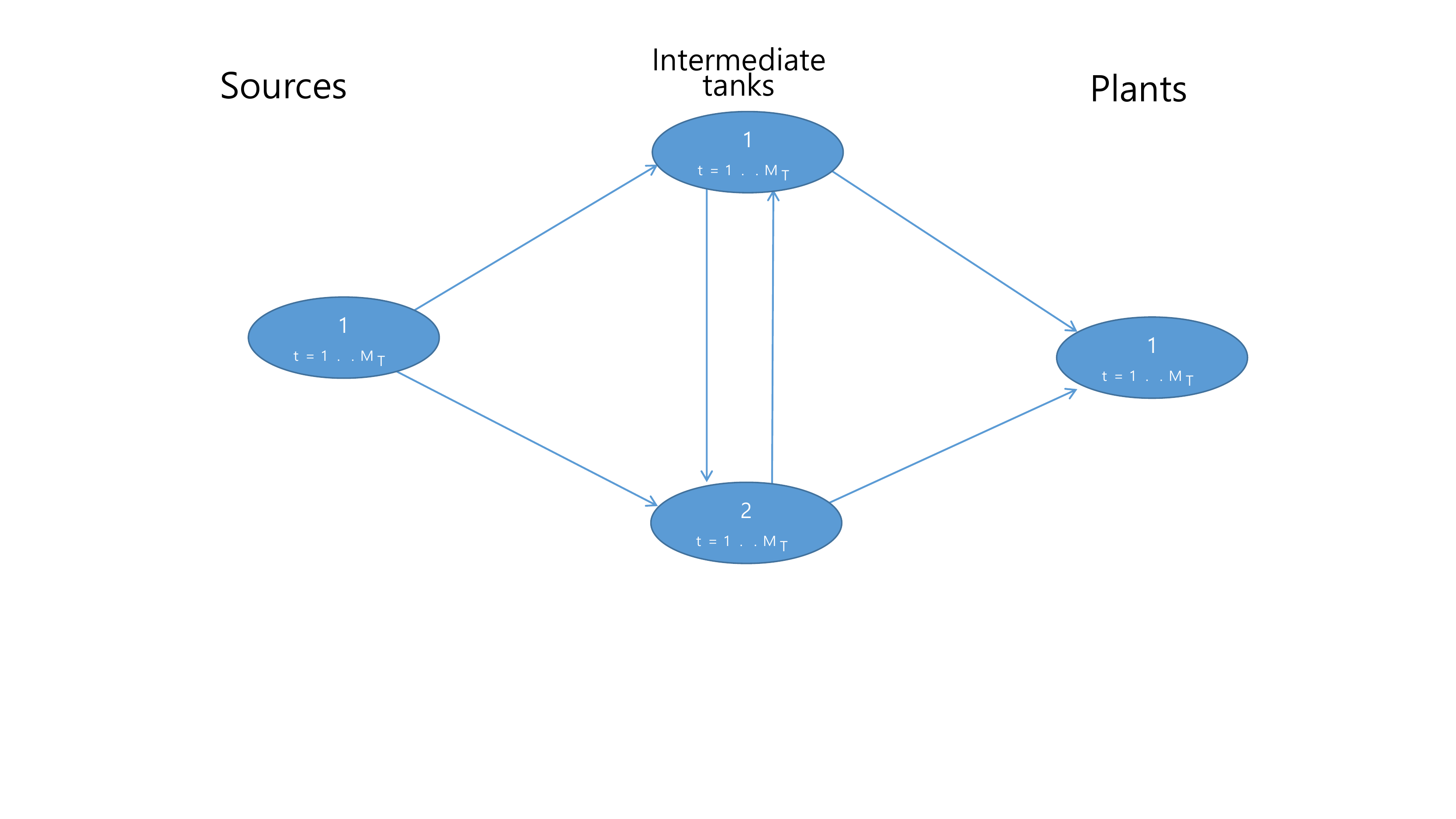} 
\hspace*{-1.7cm} 
\includegraphics[width=10cm, height=7cm]{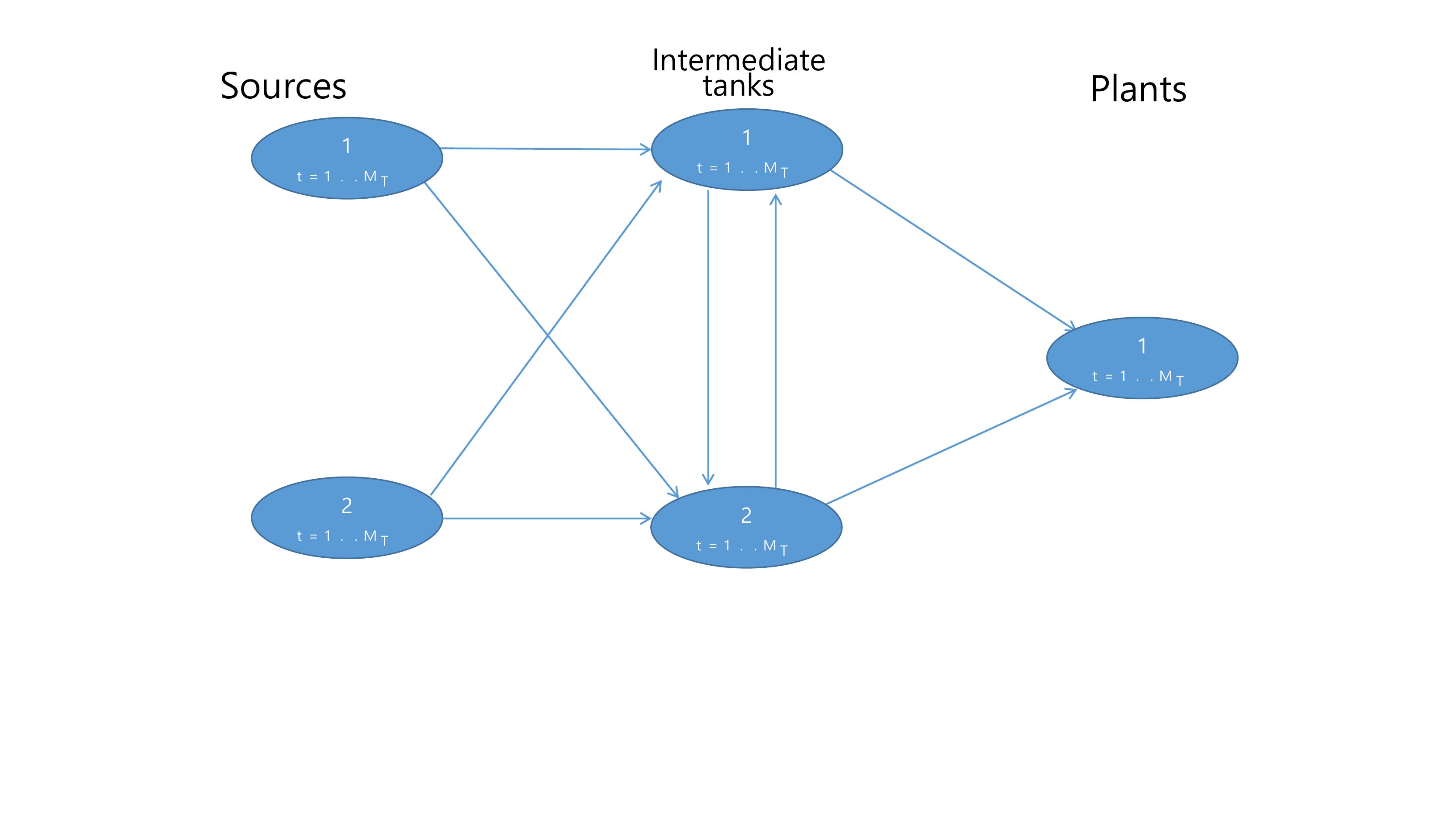}  
\vspace*{-2.3cm}
\caption{Instances 1, 2, 3, and 4}\label{F-1}
\end{figure}

\begin{figure}[hbt]
\hspace*{-1.7cm}
\includegraphics[width=11cm, height=7cm]{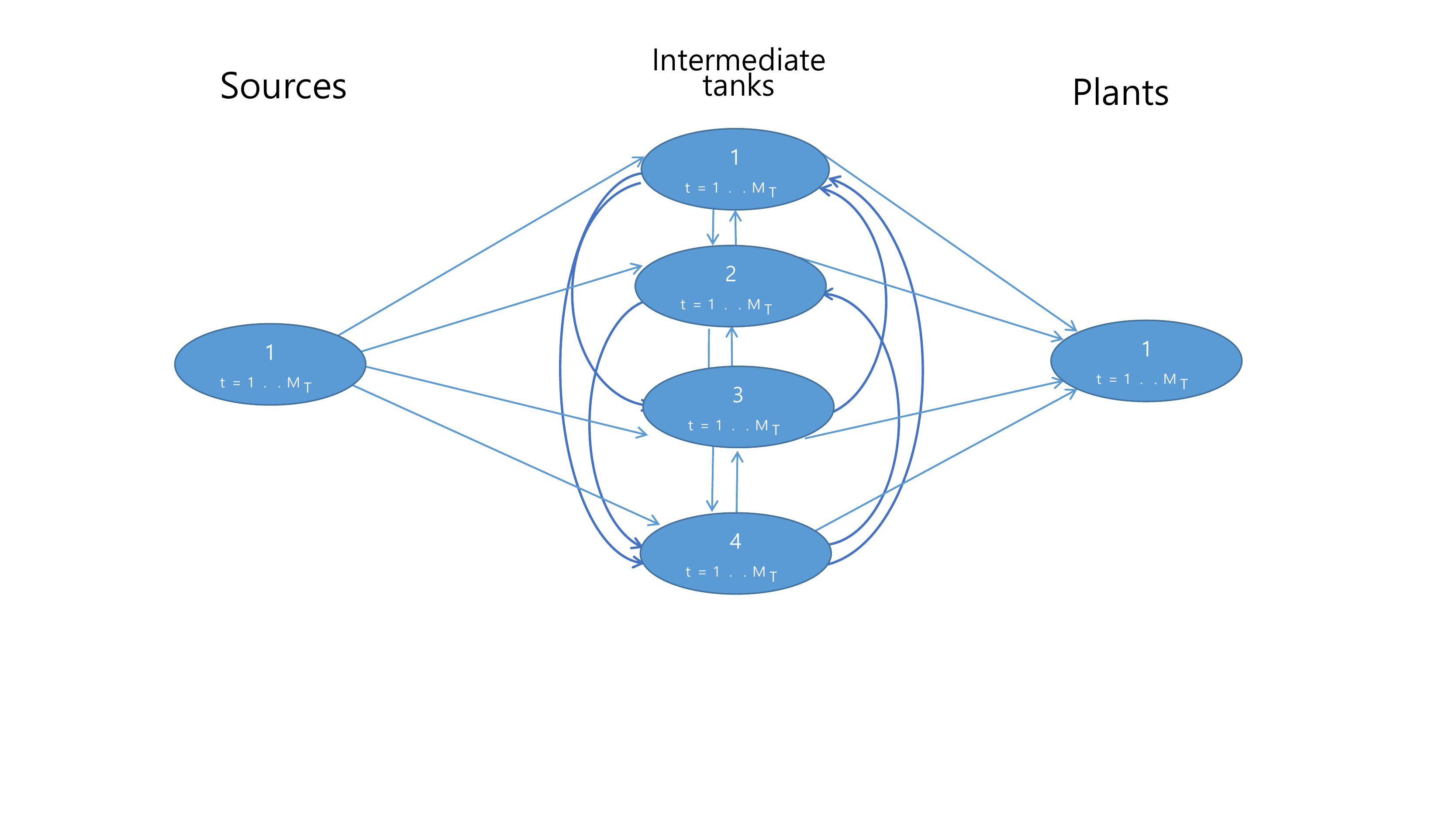} 
\hspace*{-1.7cm} 
\includegraphics[width=10cm, height=7cm]{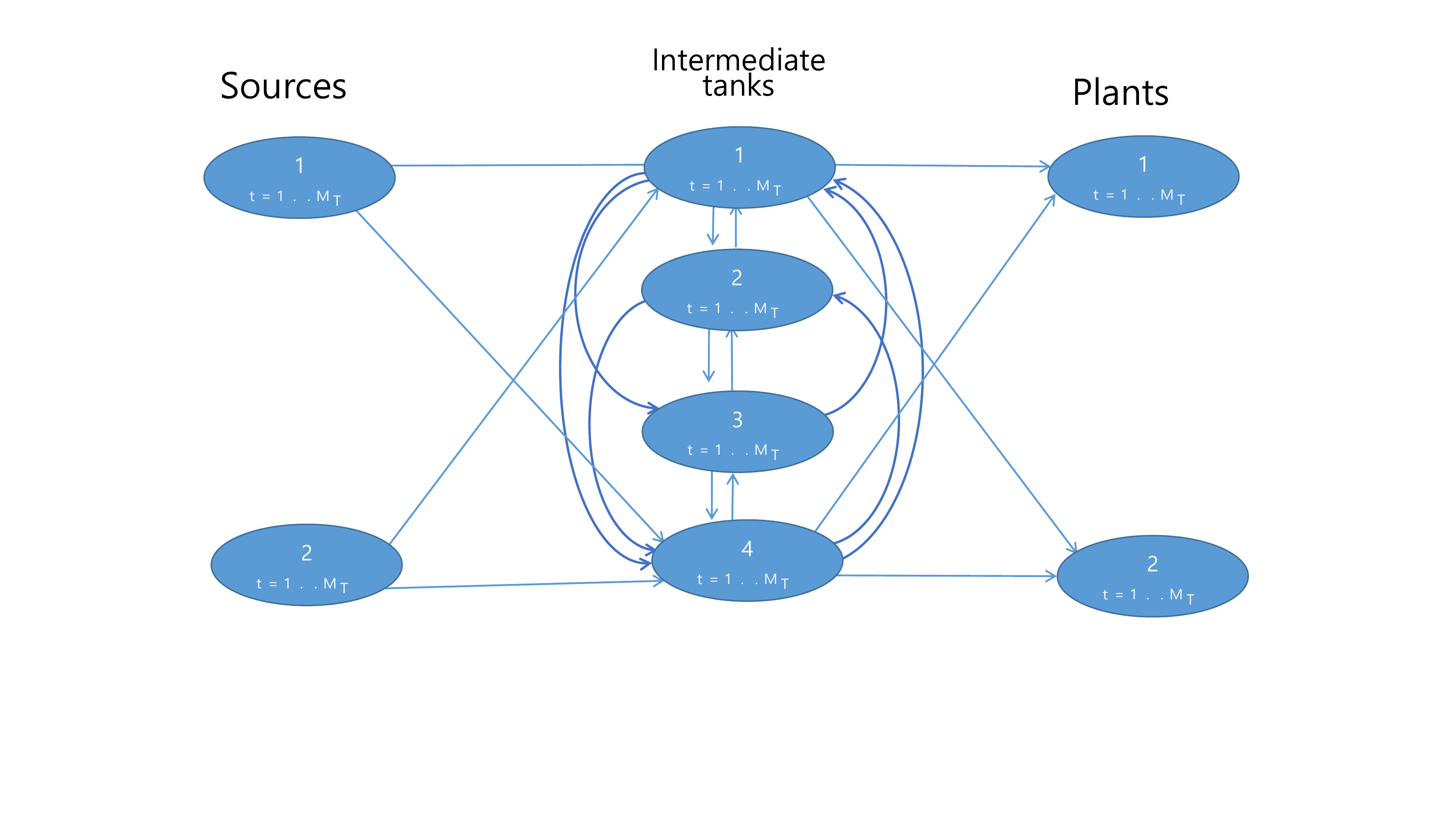}  
\vspace*{-2cm}
\caption{Instances 5, 6, 7, and 8}\label{F-2}
\end{figure}

\begin{table}[h!t]
\begin{center}
\caption{$M_S$: the number of sources, $M_I$: the number of intermediate tanks, $M_P$: the number of plants,
$(\#a, \#p , \#q, \#v)$: the numbers of variables,  $|A|$: the number of pipelines, $n$:  sum of all variables. 
} 
\label{T-3}
  \begin{tabular}{|c||c|c|c|c|c|c|c|c|c|c|} 
  \hline
Instance&$M_S$&$M_I$&$M_P$&$M_T$&$|A|$&$\#a$&$\#p$&$\#q$&$\#v$&$n$  \\ \hline
1& 1&2 &1  &10 &  6&  60&29&39&10&138 \\ \hline
2& 1&2 &1  &20 &  6&120&59&79&20&298 \\ \hline
3& 2&2 &1  &10 &  8&  80&38&48&10&176 \\ \hline
4& 2&2 &1  &20 &  8&160&78&98&10&356 \\ \hline
5& 1&4 &1  &7   &20&140&34&41&7&222 \\ \hline
6& 1&4 &1  &14 &20&280&69&83&14&446 \\ \hline
7& 2&4 &2  &28 &28&784&166&222&56&1228\\ \hline
8& 2&4 &2  &28 &28&784&166&222&56&1228 \\ \hline \hline
9&10&18&7&2 &612&1224&46&60&14&1344\\ \hline
10 &8&20&10&2 &740&1480&48&68&20&1616 \\ \hline \hline
Foulds 3&11&8&16&1&160&160&0&8&0&168 \\ \hline
Foulds 4&11&8&16&1&160&160&0&8&0&168 \\ \hline
Foulds 5&11&4&16&1&96&96&0&4&0&100 \\ \hline
  \end{tabular}
  \end{center}
\end{table}

The experimental results of each instance are summarized in Tables \ref{result-1-4}, \ref{result-5-8} and \ref{resultLarge}.
We use the following ratio to measure how much  the obtained solution successfully satisfies the requirements of all nodes:
 \begin{eqnarray*}
\mbox{sucs.ratio} \displaystyle  
&=& \left(\displaystyle \sum_{t \in T}\sum_{i \in V_{P}} RQ_{i}^{t}  
- \sum_{t \in T}\sum_{i \in V_{P}} v_{i}^{t} \right)\left/ \displaystyle \sum_{t \in T}\sum_{i \in V_{P}} RQ_{i}^{t} \right. \\
&=& \frac{ {\text{(the sum of the  requirements)}}
- {\text{(the sum of insuffcienct values)}} }{{\text{(the sum of requirements)}}}.
\end{eqnarray*}
Note that sucs.ratio in the rescheduling method is less than $100\%$
only when the rescheduling method terminates at Step 3.6.(II)2.,  as it cannot meet the quality requirements.
In the tables,   we show sucs.ratio, the optimal  values of the relaxation methods, the computed objectives values of the test instances,
  execution time for each relaxation, and {\tt fmincon} or the rescheduling method in Section 5.
 Sdp.ffs.nls means applying three procedures: (i) the SDP relaxation, (ii) FFS and (iii) {\tt fmincon}. Socp and Lp mean the SOCP and LP relaxation, respectively. Socp.ffs.reschd and Lp.ffs.reschd denote that the instances were solved by applying
 the rescheduling method including the SOCP or LP relaxation, and FFS as described in
 Section 5.  
We note that Nishi's method employs the SDP relaxation of \eqref{PPM}. The SDP relaxation in Sdp.ffs.nls is obtained from \eqref{PPO}.
 As the number of constraints in \eqref{PPO} is larger than that of \eqref{PPM},  it takes longer to solve the SDP relaxation of
  \eqref{PPO} than that of Nishi's method.
 In the column ``Ffs.fmincon or Ffs.reschd'',   the computation time 
by FFS and fmincon for the methods that employ fmincon is shown.
 FFS took very short time and  most of CPU time was consumed by fmincon in the experiments.
For the rescheduling method, the computational time  
for the entire rescheduling method excluding the conic relaxation is shown.

Table \ref{result-1-4} displays the results on the test instances 1-4 in \cite{NISHI10} whose number of variables varies
 from $n=138$ to $n=356$.
For all instances, we see in the column  Total  that the SOCP and LP relaxation consumed much shorter CPU time than Nishi's method and 
the SDP relaxation of \eqref{PPO}.  The rescheduling method took much shorter CPU time than the  nonlinear program solver {\tt fmincon}. 
As a result, Socp.ffs.reschd and Lp.ffs.reschd show shorter total CPU time than the other methods, except for the instance 2.
For the instance 2, the optimal value could be found by the SDP relaxation of   \eqref{PPM}, thus,  {\tt fmincon} did not take long to 
converge with the stopping criteria.
In the column Relax, the CPU time by Socp.ffs.reschd and Lp.ffs.reschd is longer than that of Socp.ffs.nls and Lp.ffs.nls as
the SOCP and LP relaxations are repeatedly solved in the rescheduling method as described in Algorithm 5.3.

For the objective values obtained by the methods in Table \ref{result-1-4}, we confirm that the SDP, SOCP and LP relaxations of
 \eqref{PPO}  compute the equivalent objective values, as described in Section 4.
 For Socp.ffs.reschd and Lp.ffs.reschd, two values are shown for the column of Relax.obj.val to denote the objective value at the
 starting time  and the final time, respectively, as the SOCP and LP relaxations are repeatedly solved.
Socp.ffs.reschd  and Lp.ffs.reschd frequently provide 100\% sucs.ratio with the smallest objective values in the column of Obj.val.
We observe that   Socp.ffs.reschd and Lp.ffs.reschd are computationally efficient and effective to obtain  smaller objective values
than the other methods.

\begin{table}[tbp]
\begin{center}
\caption{The results for the instances 1, 2, 3, and 4.  ``more than 24 hours'' means that the method could not
provide a solution within 24 hours.} \label{result-1-4}
\scalebox{0.9}{
\begin{tabular}{|r||r|r|r||r|r|r|} 
\hline
\multicolumn{1}{|c||}{Problem} & \multicolumn{6}{c|}{Instance 1 ($n=138$)} \\ \hline
 \multicolumn{1}{|c||}{ }   & & & & \multicolumn{3}{c|}{CPU time (seconds) }  \\ \cline{5-7}
\multicolumn{1}{|c||}{Methods} &Sucs.ratio&Relax.obj.val &  \multicolumn{1}{|c||}{Obj.val} &Relax &\thead{Ffs.fmincon or \\Ffs.reschd} &\multicolumn{1}{|c|}{Total}  \\ \hline
Nishi's  method &$97.61\%$&$202.3$&$11243.44$&6.95&1925.17&$1932.12$ \\ \hline%
Sdp.ffs.nls&$95.70\%$&$222.5$&$19827.98$&$142.67$&$340.28$&$482.95$ \\ \hline%
Socp.ffs.nls&$97.83\%$&$222.5$&$10262.53$&$12.51$&$218.38$&$230.89$ \\ \hline%
Lp.ffs.nls&$95.27\%$&$222.5$&$21792.29$&$11.76$&$5.96$&$17.72$ \\ \hline%
Socp.ffs.reschd&$100.00\%$&$222.5,\ 122.5$&$470.00$&$30.46$&$2.05$&$32.51$ \\ \hline
Lp.ffs.reschd&$98.34\%$&$222.5,\ 162.5$&$7944.94$&$20.29$&$1.25$&$21.54$ \\ \hline \hline
\multicolumn{1}{|c||}{Problem} & \multicolumn{6}{c|}{Instance 2 ($n=298$)} \\ \hline
&Sucs.ratio&Relax.obj.val &Obj.val&Relax &\thead{Ffs.fmincon or \\Ffs.reschd} &\multicolumn{1}{|c|}{Total}   \\ \hline
Nishi's  method&$99.78\%$&$617.92$&$3058.10$&8.83&16.30&$25.13$ \\ \hline%
Sdp.ffs.nls& - & - & -& \multicolumn{3}{c|}{ more than $ 24$ hours } \\ \hline
Socp.ffs.nls&$99.74\%$&$800.00$&$3407.69$&$40.60$&$15.74$&$56.34$ \\ \hline%
Lp.ffs.nls&$99.87\%$&$800.00$&$2253.85$&$29.54$&$15.73$&$45.28$ \\ \hline%
Socp.ffs.reschd&$100.00\%$&$800.00 ,\ 687.06$&$1092.94 $&$75.15$&$1.99$&$77.14$ \\ \hline
Lp.ffs.reschd&$100.00\%$&$800.00,\ 687.06$&$1092.94$&$61.69$&$1.92$&$63.61$ \\ \hline \hline
\multicolumn{1}{|c||}{Problem} & \multicolumn{6}{c|}{Instance 3 ($n=176$)} \\ \hline
&Sucs.ratio&Relax.obj.val &Obj.val&Relax &\thead{Ffs.fmincon or \\Ffs.reschd} &\multicolumn{1}{|c|}{Total}   \\ \hline
Nishi's  method&$98.66\%$&$181.14$&$6594.55 $&$7.82$&$2866.87$&$2874.69$\\ \hline%
Sdp.ffs.nls&$99.99\%$&$300.00$&$546.91$&$548.06$&$2856.68$&$3404.74 $ \\ \hline%
Socp.ffs.nls&$80.11\%$&$300.00$&$89974.73$&$16.21$&$2885.36$&$2901.57 $ \\ \hline%
Lp.ffs.nls&$99.99\%$&$300.00$&$510.81$&$14.17$&$2892.56 $&$2906.73 $ \\ \hline%
Socp.ffs.reschd&$100.00\%$&$300.00 ,\ 91.57$&$503.97$&$36.49$&$1.62$&$38.11$ \\ \hline
Lp.ffs.reschd&$99.70\%$&$300.00,\ 55.57$&$1852.02 $&$32.75$&$1.78$&$34.52$ \\ \hline \hline
\multicolumn{1}{|c||}{Problem} & \multicolumn{6}{c|}{Instance 4 ($n=356$)} \\ \hline
&Sucs.ratio&Relax.obj.val &Obj.val&Relax &\thead{Ffs.fmincon or \\Ffs.reschd} &\multicolumn{1}{|c|}{Total}   \\ \hline
Nishi's  method&$96.06\%$&$728.18$&$36635.41$&13.20&7044.23&$7057.45 $\\ \hline%
Sdp.ffs.nls& - & - & -& \multicolumn{3}{c|}{ more than $ 24$ hours } \\ \hline
Socp.ffs.nls&$97.03\%$&$900.00$&$27842.37$&$77.87$&$4294.91$&$4372.78 $ \\ \hline%
Lp.ffs.nls&$99.48\%$&$900.00$&$5758.37$&$58.74$&$7334.06$&$7392.80 $ \\ \hline%
Socp.ffs.reschd&$100.00\%$&$900.00,\ 41.87$&$1076.34 $&$186.49$&$4.28$&$190.77$ \\ \hline%
Lp.ffs.reschd&$100.00\%$&$900.00,\ 170.71$&$1062.41$&$161.04$&$5.71$&$166.75$ \\ \hline%
\end{tabular}
}
\end{center}
\end{table}
Table \ref{result-5-8} shows the results for the instances 5, 6, 7, and 8 in \cite{NISHI10} where
$n$ ranges from 222 to 1228.
We also see in the column  Total  that the CPU time spend by Socp.ffs.reschd and Lp.ffs.reschd is much shorter than the other methods.
In the column of Obj.val, the smallest objective values were obtained by Lp.ffs.reschd except for the instance 5.
We notice that the highest sucs.ratio leads to the smallest objective value for all instances.
For the instances 7 and 8 where $n$ is large, Socp.ffs.reschd and Lp.ffs.resched achieve 100\% sucs.ratio, and
Lp.ffs.reschd consumed the shortest CPU time. 

\begin{table}[tbp]
\begin{center}
\caption{The results for the instances 5, 6, 7, and 8. ``more than 24 hours'' means that the method could not
provide a solution within 24 hours.} \label{result-5-8}
\scalebox{0.9}{
\begin{tabular}{|r||r|r|r||r|r|r|} 
\hline
\multicolumn{1}{|c||}{Problem} & \multicolumn{6}{c|}{Instance 5 ($n=222$)} \\ \hline
 \multicolumn{1}{|c||}{ }   & & & & \multicolumn{3}{c|}{CPU time (seconds) }  \\ \cline{5-7}
\multicolumn{1}{|c||}{Methods} &Sucs.ratio&Relax.obj.val &  \multicolumn{1}{c||}{Obj.val} & \multicolumn{1}{c|}{Relax} &\thead{Ffs.fmincon or \\Ffs.reschd} &\multicolumn{1}{|c|}{Total}  \\ \hline
Nishi's  method&$99.44\%$&$17.87$&$3046.25$&13.33&185.67&$198.99$ \\ \hline%
Sdp.ffs.nls&$92.59\%$&$134.00$&$35658.10$&$3239.96$&$14.54$&$3254.50$ \\ \hline%
Socp.ffs.nls&$91.60\%$&$134.00$&$40381.10$&$19.27$&$153.36$&$172.63$ \\ \hline%
Lp.ffs.nls&$92.76\%$&$134.00$&$34890.66$&$17.81$&$786.83$&$804.64$ \\ \hline%
Socp.ffs.reschd&$92.62\%$&$134.00,\ 54.00$&$35483.11$&$29.10$&$1.73$&$30.84$ \\ \hline%
Lp.ffs.reschd&$94.34\%$&$134.00,\ 54.00$&$27358.17$&$30.07$&$1.60$&$31.67$ \\ \hline \hline%
\multicolumn{1}{|c||}{Problem} & \multicolumn{6}{c|}{Instance 6 ($n=446$)} \\ \hline
&Sucs.ratio&Relax.obj.val & \multicolumn{1}{c||}{Obj.val} & \multicolumn{1}{c|}{Relax} &\thead{Ffs.fmincon or \\Ffs.reschd} &\multicolumn{1}{|c|}{Total}   \\ \hline
Nishi's  method&$95.74\%$&$186.81$&$41470.49$&33.55&18000.13&$18033.68$ \\ \hline%
Sdp.ffs.nls& - & - & -& \multicolumn{3}{c|}{ more than $ 24$ hours } \\ \hline%
Socp.ffs.nls&$93.14\%$&$334.00$&$66194.84$&$102.33$&$1055.38 $&$1157.71 $ \\ \hline%
Lp.ffs.nls&$96.51\%$&$334.00$&$34102.63$&$80.02$&$9738.29$&$9818.31$ \\ \hline%
Socp.ffs.reschd&$96.11\%$&$334.00,\ 134.00$&$37889.83$&$189.82$&$8.90$&$198.72$ \\ \hline%
Lp.ffs.reschd&$97.85\%$&$334.00,\ 38.37$&$21272.54$&$226.01$&$8.09$&$234.10$ \\ \hline \hline%
\multicolumn{1}{|c||}{Problem} & \multicolumn{6}{c|}{Instance 7 ($n=1228$)} \\ \hline
&Sucs.ratio&Relax.obj.val & \multicolumn{1}{c||}{Obj.val} & \multicolumn{1}{c|}{Relax} &\thead{Ffs.fmincon or \\Ffs.reschd} &\multicolumn{1}{|c|}{Total}   \\ \hline
Nishi's  method&$99.51\%$&$475.33$&$7176.40$&$263.59$&$39716.96$&$39980.55$ \\ \hline%
Sdp.ffs.nls& - & - & -& \multicolumn{3}{c|}{ more than $ 24$ hours } \\ \hline
Socp.ffs.nls&$99.99\%$&$824.00$&$1649.13$&$1754.10$&$32448.45 $&$34202.55$ \\ \hline%
Lp.ffs.nls&$99.99\%$&$824.00$&$1667.51$&$1368. 04$&$9707.45$&$11075.49$ \\ \hline%
Socp.ffs.reschd&$100.00\%$&$824.00,\ 1916.96$&$1605.70$&$11761.98$&$674.45$&$12436.43$ \\ \hline%
Lp.ffs.reschd&$100.00\%$&$824.00,\ 1996.89$&$1599.44$&$9177.46$&$688.87$&$9866.32$ \\ \hline \hline%
\multicolumn{1}{|c||}{Problem} & \multicolumn{6}{c|}{Instance 8 ($n=1228$)} \\ \hline
&Sucs.ratio&Relax.obj.val & \multicolumn{1}{c||}{Obj.val} & \multicolumn{1}{c|}{Relax}  &\thead{Ffs.fmincon or \\Ffs.reschd} &\multicolumn{1}{|c|}{Total}   \\ \hline
Nishi's  method&$98.27\%$&$458.22$&$22214.26$&$256.89$&$39443.43$&$39700.32$ \\ \hline%
Sdp.ffs.nls& - & - & -& \multicolumn{3}{c|}{ more than $ 24$ hours } \\ \hline%
Socp.ffs.nls&$99.99\%$&$824.00$&$1714.70$&$1746.94$&$20929.01$&$22675.94$ \\ \hline%
Lp.ffs.nls&$99.33\%$&$824.00$&$8957.09$&$1368.09$&$39583.53$&$40951.62$ \\ \hline%
Socp.ffs.reschd&$100.00\%$&$824.00,\ 20.07$&$1636.82 $&$13311.83$&$800.91$&$14112.74$ \\ \hline%
Lp.ffs.reschd&$100.00\%$&$824.00, \ 20.07$&$1634.99$&$8835.86$&$660.60$&$9496.46 $ \\ \hline%
\end{tabular}
}
\end{center}
\end{table}
\begin{table}[htbp]
\begin{center}
\caption{The results for the instances  9 and 10. The instances 9 and 10 with $M_T = 2$.
 ``$>$ 24 hours'' means that the method could not
provide a solution within 24 hours.} \label{resultLarge}
\scalebox{0.9}{
\begin{tabular}{|r||r|r|r||r|r|r|} 
\hline
\multicolumn{1}{|c||}{Problem} & \multicolumn{6}{c|}{Instance 9 ($n=1344$)} \\ \hline
 \multicolumn{1}{|c||}{ }   & & & & \multicolumn{3}{c|}{CPU time (seconds) }  \\ \cline{5-7}
&Sucs.ratio&Relax.obj.val &Obj.val&Relax &\thead{Ffs.fmincon or \\Ffs.reschd} &\multicolumn{1}{|c|}{Total}   \\ \hline
Nishi's  method&\multicolumn{6}{c|}{Fail to solve the SDP relaxation due to out-of-memory} \\ \hline%
Socp.ffs.nls&$82.83\%$&$490.05$&$81347.30$&$805.78$& $> 24$ hours&$-$ \\ \hline%
Lp.ffs.nls&$91.39\%$&$490.05$&$41146.81$&$588.55$&$ > 24$ hours&$-$ \\ \hline%
Socp.ffs.reschd&$100.00\%$&$490.05 ,\ 95.12$&$657.35$&$919.18$&$112.92$&$1032.10$ \\ \hline
Lp.ffs.reschd&$100.00\%$&$490.05,\ 94.77$&$657.35$&$674.68$&$112.96$&$787.64$ \\ \hline \hline
\multicolumn{1}{|c||}{Problem} & \multicolumn{6}{c|}{Instance 10 ($n=1616$)} \\ \hline
 \multicolumn{1}{|c||}{ }   & & & & \multicolumn{3}{c|}{CPU time (seconds) }  \\ \cline{5-7}
&Sucs.ratio&Relax.obj.val &Obj.val&Relax &\thead{Ffs.fmincon or \\Ffs.reschd} &\multicolumn{1}{|c|}{Total}   \\ \hline
Nishi's  method&\multicolumn{6}{c|}{Fail to solve the SDP relaxation due to out-of-memory}\\ \hline%
Socp.ffs.nls&$86.00\%$&$511.00$&$105842.11$&$1236.56$& $ > 24$ hours&$-$ \\ \hline%
Lp.ffs.nls&$89.53\%$&$511.00$&$81377.36$&$911.34$&$ > 24$ hours&$-$ \\ \hline%
Socp.ffs.reschd&$100.00\%$&$511.00 ,\ 108.87$&$759.06$&$1425.09$&$28.29$&$1453.38$ \\ \hline
Lp.ffs.reschd&$100.00\%$&$511.00,\ 102.26$&$759.06$&$1050.38$&$30.07$&$1080.45$ \\ \hline \hline
\end{tabular}
}
\end{center}
\end{table}
Next,
we tested the proposed method
on the problems with the minimum number of time discretization to see how large number of variables 
can be solved with the proposed method. 
The number of sources, intermediate tanks, and plants are shown in Table \ref{T-3} and the intermediate tanks
in the instances 9 and 10 are connected  as a complete graph shown in Figure \ref{F-2}.
The SDP relaxations for Nishi's method and \eqref{SDP} could not be solved since the sizes of the SDP relaxations
were too large to handle on our computer. In Table \ref{resultLarge}, we observe that
{\tt fmincon} used in the methods Socp.ffs.nls and Lp.ffs.nls took long time, and could not provide a solution
within 24 hours for the  instances  9 and 10. 
On the other hand, the rescheduling method,  Socp.ffs.reschd and Lp.ffs.reschd, successfully
solve the problems in much shorter time.
The objective values obtained by Socp.ffs.reschd and Lp.ffs.reschd are smaller than those of Socp.ffs.nls and Lp.ffs.nls,
providing 100\% sucs.ratio.

We tested the proposed method on Foulds 3, 4 and 5 which have the largest number of
variables among the standard test problems for
the pooling problem in P-formulation such as Haverly, Ben-Tal, Foulds, Adhya, and RT2. 
As time discretization is not used in the problems, the SDP, SOCP and LP relaxations and FFS are applied to Foulds 3, 4 and 5. 
Table \ref{F-6} displays the numerical results on the test problems. 
We see that the objective value obtained by the SDP, SOCP and LP relaxations are equivalent.
FFS was applied to find a feasible solution of the original pooling problem.
From the CPU time spent by the SDP, SOCP and LP relaxations,  we observe that the LP relaxation is most efficient.
The optimal values of Foulds 3, 4 are known as -8 \cite{MORANDI18}. 
The proposed method  finds an approximate solution of Foulds 3, 4 with $n=168$  with much shorter computational
time than that in \cite{MORANDI18}. 
As the instances 9 and10  are much larger than Foulds 3, 4 and 5, the proposed method can be applied to larger pooling
problems than the standard pooling test problems.

\begin{table}[htbp]
\begin{center}
\caption{ Numerical results on Foulds problems} \label{F-6}
\begin{tabular}{|c|r|r||r|r|r|} 
\hline
Problem &\multicolumn{5}{c|}{ Foulds 3 $(n=168)$} \\ \hline
\multirow{ 2}{*}{Relaxation}&\multicolumn{2}{c||}{ }& \multicolumn{3}{c|}{CPU time (seconds)} \\ \cline{2-6}
   &\multicolumn{1}{c|}{Relax.obj.val } &\multicolumn{1}{c||}{FFS.obj.val}&\multicolumn{1}{c|}{ Relax.} & \multicolumn{1}{c|}{ Ffs }&\multicolumn{1}{c|}{ Total} \\ \hline
Sdp.ffs  &$-9.00$&$-3.82$&$317.63$&$0.81$&$318.44$ \\ \hline
Socp.ffs &$-9.00$&$-3.82$&$9.81$&$0.45$&$10.26$ \\ \hline
Lp.ffs &$-9.00$&$-3.82$&$8.83$&$0.43$&$9.26$ \\ \hline \hline
Problem & \multicolumn{5}{c|}{ Foulds 4 $(n=168)$} \\ \hline
   &\multicolumn{1}{c|}{Relax.obj.val } &\multicolumn{1}{c||}{FFS.obj.val}&\multicolumn{1}{c|}{  Relax.} & \multicolumn{1}{c|}{ Ffs }&\multicolumn{1}{c|}{ Total} \\ \hline
Sdp.ffs &$-9.00$&$-2.64$&$300.51$&$0.44$&$300.95$ \\ \hline
Socp.ffs &$-9.00$&$-2.64$&$10.02$&$0.44$&$10.46$ \\ \hline
Lp.ffs &$-9.00$&$-2.64$&$9.24$&$0.44$&$9.68$ \\ \hline \hline
Problem & \multicolumn{5}{c|}{ Foulds 5 $(n=100)$} \\ \hline
   &\multicolumn{1}{c|}{Relax.obj.val } &\multicolumn{1}{c||}{FFS.obj.val}&\multicolumn{1}{c|}{  Relax.} & \multicolumn{1}{c|}{ Ffs }&\multicolumn{1}{c|}{ Total} \\ \hline
Sdp.ffs &$-11.00$&$-0.83$&$23.47$&$0.45$&$23.92$ \\ \hline
Socp.ffs &$-11.00$&$-0.83$&$6.56$&$45$&$7.01$ \\ \hline
Lp.ffs &$-11.00$&$-0.83$&$5.97$&$0.47$&$6.42$ \\ \hline
\end{tabular}
\end{center}
\end{table}

\section{Concluding remarks}

We have proposed an efficient computational method for the pooling problem with time discretization using
 the SOCP and LP relaxations and the rescheduling method. From the form of QCQPs for the 
 pooling problem in \cite{NISHI10},
our formulation with time discretization has been obtained by relaxing the equality constraints
into inequality constraints and introducing penalty terms in the objective function.
We have shown theoretically that there exists no gap among the optimal values of the SDP, SOCP and LP relaxations of 
our formulation. This theoretical result can be used in other formulations of the pooling problem where only bilinear terms
appear.

Computational results have been presented to show the efficiency of the proposed method over the SDP relaxations of
the pooling problems and applying a nonlinear programming solver.
From the numerical results, we have demonstrated that the SOCP and LP relaxations are more computationally efficient
than the SDP relaxations while obtaining the same optimal value. Moreover, our proposed rescheduling method is much faster  than the nonlinear programming solver 
{\tt fmincon} and effective in obtaining a solution that satisfies all the requirements. As a result, large instances up to $n=1616$
could be solved with the LP relaxation and the rescheduling method.

As  the  pooling problem  is a bilinear problem, it may be possible to utilize the structure  \cite{MORANDI18}
of the problem to further improve
the computational efficiency. We hope to investigate the underlying structure of the pooling problem 
for solving  large-scale   problems.

\bibliographystyle{plain}

\end{document}